\documentclass[reqno,a4paper,11pt]{amsart}
\usepackage[english]{babel}
\usepackage{amssymb}
\usepackage{amsmath}
\usepackage{MnSymbol}
\usepackage[all]{xy}
\usepackage[a4paper,left=4cm,right=4cm]{geometry}
\usepackage{graphicx}
\usepackage{mathtools}
\usepackage{enumerate}
\usepackage{braket}
\usepackage{url}
\usepackage{tikz}
\usepackage{tikz-cd}
\tikzcdset{diagrams={arrows={shorten >=-0.5ex,shorten <=-0.5ex}}}
\usetikzlibrary{graphs}
\usetikzlibrary{arrows,chains,matrix,positioning,scopes}
\makeatletter
\tikzset{joinlabel/.code=\tikzset{after node path={%
\ifx\tikzchainprevious\pgfutil@empty\else(\tikzchainprevious)%
edge[every join]#1(\tikzchaincurrent)\fi}}}

\makeatother

\newtheorem{prop}{Proposition}
\newtheorem{thm}[prop]{Theorem}
\newtheorem{cor}[prop]{Corollary}
\newtheorem{lem}[prop]{Lemma}
\newtheorem{fact}[prop]{Fact}

\theoremstyle{definition}
\newtheorem{defn}[prop]{Definition}
\newtheorem{ques}[prop]{Question}
\newtheorem{ex}[prop]{Example}

\theoremstyle{remark}
\newtheorem{rem}[prop]{Remark}

\numberwithin{prop}{section}
\numberwithin{equation}{section}

\newcommand{\biB}{\mathrm{Bi}}

\newcommand{\Hom}{\mathrm{Hom}}
\newcommand{\Ext}{\mathrm{Ext}}
\newcommand{\dExt}{\mathrm{dExt}}
\newcommand{\dTor}{\mathrm{dTor}}
\newcommand{\dH}{\mathrm{dH}}
\newcommand{\dd}{\mathrm{d}}
\newcommand{\ff}{\mathrm{f}}
\newcommand{\iid}{\mathrm{id}}
\newcommand{\stab}{\mathrm{stab}}

\newcommand{\ccd}{\mathrm{cd}}
\newcommand{\hhd}{\mathrm{hd}}
\newcommand{\pd}{\mathrm{pr.dim}}
\newcommand{\fd}{\mathrm{fl.dim}}

\newcommand{\ob}{\mathrm{ob}}

\newcommand{\Aut}{\mathrm{Aut}}

\newcommand{\ind}{\mathrm{ind}}
\newcommand{\rst}{\mathrm{res}}

\newcommand{\FP}{\textup{FP}}
\newcommand{\F}{\textup{F}}
\newcommand{\KP}{\textup{KP}}
\newcommand{\K}{\textup{K}}
\newcommand{\ca}[1]{\mathcal{#1}}
\newcommand{\eu}[1]{\mathfrak{#1}}
\newcommand{\CO}{\mathfrak{CO}}
\newcommand{\Z}{\mathbb{Z}}
\newcommand{\Q}{\mathbb{Q}}
\newcommand{\N}{\mathbb{N}}

\newcommand{\R}{\mathbb{R}}
\newcommand{\RR}{\textup{R}}

\newcommand{\gwr}{\nwsebipropto}

\newcommand{\RG}{\RR[G]}
\newcommand{\RH}{\RR[H]}

\newcommand{\QG}{\Q[G]}
\newcommand{\QH}{\Q[H]}

\newcommand{\dis}{\mathbf{dis}}
\newcommand{\Mod}{\mathbf{mod}}
\newcommand{\Rmod}{{}_\RR\Mod}

\newcommand{\RGdis}{{}_{\RR[G]}\dis}
\newcommand{\RGmod}{{}_{\RR[G]}\Mod}

\newcommand{\QGdis}{{}_{\QG}\dis}

\newcommand{\RGtop}{{}_{\RR[G]}\mathbf{top}}

\newcommand{\disQG}{\dis_{\QG}}
\newcommand{\disRG}{\dis_{\RG}}
\newcommand{\modRG}{\Mod_{\RG}}
\newcommand{\Qmod}{{}_{\Q}\Mod}

\newcommand{\colim}{\mathrm{colim}}
\newcommand{\caR}{\ca{R}}

\newcommand{\quot}{\mathbin{\!/\mkern-5mu/\!}}
\newcommand{\argu}{\hbox to 7truept{\hrulefill}}

\title[Finiteness properties of t.d.l.c. group]{Finiteness properties of totally disconnected locally compact groups}

\author[I. Castellano]{I. Castellano}
\address{Department of Mathematics, Computer Science and Physics - University of Udine, Via delle Scienze 206, Udine \textup{33100}, Italy.}
\email{ilaria.castellano88@gmail.com}

\author[G. Corob Cook]{G. Corob Cook}
\address{Department of Mathematics -
University of the Basque Country UPV/EHU,
Leioa (Bilbao), Spain}
\email{gcorobcook@gmail.com}

\begin{document}

\begin{abstract}
In this paper we investigate finiteness properties of totally disconnected locally compact groups for general commutative rings $\textup{R}$, in particular for $\RR = \Z$ and $\RR = \Q$. We show these properties satisfy many analogous results to the case of discrete groups, and we provide analogues of the famous Bieri's and Brown's criteria for finiteness properties and deduce that both $\FP_n$-properties and  $\F_n$-properties are quasi-isometric invariant. Moreover, we introduce graph-wreath products in the category of totally disconnected locally compact groups and discuss their finiteness properties.
\end{abstract}
\maketitle
\section*{Introduction}
\let\thefootnote\relax\footnote{This work was supported by EPSRC Grant N007328/1 Soluble Groups and Cohomology.}
The notion of compact presentation was introduced in 1964 by Knesern, and  a first attempt at generalising this to higher dimension is due to Abels and Tiemeyer \cite{AT}. However, compact presentability and related finiteness conditions for locally compact groups have received little attention until recently. In \cite{cdlh} Cornulier and de la Harpe revisited $\sigma$-compactness, finite presentability and compact presentability of locally compact groups from the point of view of coarse geometry. A homological approach to finiteness conditions of totally disconnected locally compact groups over the rationals was introduced in \cite{it:ratdiscoh}.  Nevertheless, the theory of finiteness conditions for totally disconnected locally compact (t.d.l.c.) groups is still less developed than the one for discrete groups.  

In this paper we investigate finiteness properties of t.d.l.c. groups for general commutative rings $\RR$  (and in particular for $\RR=\Z$ and $\RR=\Q$) bringing forward the study initiated in \cite{it:ratdiscoh}.  There the authors take advantage of the divisibility of $\Q$ in order to deal with a category of modules with enough projectives. When $\RR$ is arbitrary instead,  the lack of projective objects means that to get results over $\RR$, rather than just $\Q$, requires some additional work: it requires the theory of topological $\RG$-modules introduced in \cite{Ged}. The benefits of this generalisation will become apparent when we want to study $G$-spaces and $\F_n$-properties. Indeed, we need to understand the homology over $\Z$ in order to use the Hurewicz theorem to deduce information about the homotopy of these spaces from their homology; see, for example, Proposition \ref{prop:FvsFP}.

\medskip

The paper opens with a section dedicated to the categories $\RGdis$ and $\RGtop$, whose objects are respectively discrete $\RG$-modules and  topological $\RG$-modules. The objects of $\RGdis$ are the left $\RG$-modules $M$ such that the action $G\times M\to M$ is continuous if $M$ is endowed with the discrete topology. In the literature, discrete modules first appeared in the context of profinite groups \cite{serre}  and were then brought to the context of t.d.l.c. groups in \cite{it:ratdiscoh}.  

For an arbitrary ring $\RR$, the category $\RGdis$ is well-behaved: it is abelian, it has enough injectives and all (co)limits exist, but the situation is best for $\RR=\Q$. In fact, $\QGdis$ has also enough projectives and the category becomes amenable to many other homological tools. For example, a Lazard-type theorem (see Theorem~\ref{thm:lazard}) provides a complete description of all flat objects of $\QGdis$ by means of permutation modules.

The use of permutation modules is predominant in the study of finiteness conditions of t.d.l.c. groups. Of particular interest are the so-called {\it proper discrete permutation modules}, i.e., modules $\RR[\Omega]$ freely $\RR$-generated by a $G$-set $\Omega$ whose pointwise stabilisers are compact open subgroups of $G$. Proper discrete permutation modules are the key example of projective objects in $\QGdis$ which the reader may wish to keep in mind while reading the paper. Unfortunately, they fail to be projective over other rings, even in such basic cases as $R=\Z$.
 
Therefore, we embed $\RGdis$ into a quasi-abelian category that has enough projectives, whatever $\RR$ is chosen. This is possible because a t.d.l.c. group $G$ is a k-space and we can construct the category $\RGtop$ whose objects are the so-called  {\it k-$\RG$-modules}, i.e., module objects in the category of k-spaces over the k-$\RR$-algebra $\RG$; see \S\ref{ss:topmod} for more details.

\medskip 

The key technical part of this work is Section~\ref{s:comparison} where we define several finiteness properties for t.d.l.c. groups, in terms of the category $\RGtop$ as well as $\RGdis$. The terms $\FP_n$ and $\F_n$ will indicate finiteness conditions in $\QGdis$ and the terms $\KP_n$ and $\K_n$ will denote that we are in the topological context of k-spaces. For instance, a  t.d.l.c. group $G$ is defined to have {\it type $\FP_n$ over $\RR$} ($n\in\N\cup\{\infty\}$) if the trivial discrete $\RG$-module $\RR$ admits a resolution $P\to\RR$ formed by finitely generated proper discrete permutation $\RG$-modules. This definition coincides with the one in \cite{it:ratdiscoh} when $\RR=\Q$ and proper discrete permutation modules are projective. On the other hand, $G$ will be said to have {\it type $\KP_n$} if  there is a projective resolution $P \to \RR$ in $\RGtop$ with $P_i$ a free k-$\RG$-module on a compact space for $i \leq n$. Since $\RGdis$ is an abelian subcategory of $\RGtop$, we may compare these notions: having type $\FP_n$ turns out to be equivalent to having type $\KP_n$ by Theorem~\ref{fpnkpn}.
 
By analogy with the abstract case, we also consider $G$-spaces and introduce {\it type $\F_n$} and {\it  type $\K_n$} over $\RR$. These notions generalise  compact presentability to dimension $n>2$; see \S~\ref{s:comparison} for definitions and details. Theorem~\ref{thm:fnkn} shows that type $\F_n$ implies type $\K_n$ but the converse remains open.
 
 Many well-known properties that hold for abstract groups satisfying some finiteness properties can be transferred to the context of t.d.l.c. groups:
 \begin{itemize}
\item $\F_n$ (resp. $\K_n$) implies $\FP_n$ (resp. $\KP_n$) over $\RR$;
\item  $\FP_n$ implies $\F_n$ over $\Z$ if $G$ is compactly presented by Proposition~\ref{prop:FvsFP};
\item finiteness properties behave well with respect to short exact sequences of modules (see Proposition~\ref{prop:KPn} and Corollary~\ref{cor:FPn}) and extensions of t.d.l.c. groups (see Theorem \ref{thm:LHS});
\end{itemize}
and some can be even improved; see, for example, Remark~\ref{rem:improve}.
 
In addition, various famous criteria for finiteness properties -- such as Bieri's criterion and Brown's criteria -- find an analogue in the context of t.d.l.c. groups as well; see Theorem~\ref{thm:Bieri}, Theorem~\ref{thm:Brown1} and Theorem~\ref{thm:Brown3}. Indeed, the whole of Section~\ref{s:criteria} deals with all these criteria. Though these results remain true, the proofs of their discrete analogues do not always carry over directly to our situation (this occurs for Theorem~\ref{thm:Brown1}, for example), and some additional work is required.

Sections \ref{s:qii} and \ref{s:gwp} present several applications for the finiteness criteria described above. More precisely, in Section~\ref{s:qii} we answer  a question  by N. Petrosyan by proving Theorem~\ref{thm:quasiretract} and deducing that being of type $\FP_n$ (or $\F_n$) is a quasi-isometric invariant property of a t.d.l.c. group $G$. Furthermore,  we introduce in Section~\ref{s:gwp} the notion of  graph-wreath product in the category of t.d.l.c. groups and address finiteness properties for such a construction by extending results from \cite{BCK} and \cite{KropMart}.

Finally, in Section \ref{hdcd} we consider homological and cohomological dimension for t.d.l.c. groups, and prove some basic properties.
\section{A bit more on the categories $\RGtop$ and $\RGdis$}

In this paper $\RR$ will always be a commutative ring, with the discrete topology. (In practice, we can think of $\RR$ as being either $\Z$ or $\Q$.) By $G$  we will always denote a t.d.l.c. group - unless stated differently - and the subgroups of $G$ will usually be considered to be closed. The category of abstract left $\RG$-modules and their homomorphisms is denoted by $\RGmod$, while the objects of $\modRG$ are the abstract right $\RG$-modules. Subsequently, the category $\RGdis$ (resp. $\disRG$) indicates the full subcategory of $\RGmod$ (resp. $\modRG$) whose objects are the discrete left (resp. right) $\RG$-modules. Recall that an abstract $\RG$-module $M$ is said to be {\it discrete} if the action $G \times M \to M$ is continuous with the discrete topology on $M$, or equivalently if the pointwise stabiliser $G_m$, $m\in M$, is an open subgroup of $G$ for every element $m$ of $M$. The category $\RGdis$ (resp. $\disRG$) is an abelian category with enough injectives. All the upcoming properties and definitions stated for left modules find an analogue for right modules as well.

\subsection{Discrete permutation $\RG$-modules}\label{ss:perm}

Suppose $G$ acts continuously on a discrete left $G$-set $\Omega$, i.e., the stabiliser $G_\omega=\{g\in G\mid g\cdot\omega=\omega\}$ is open for all $\omega\in\Omega$. The $\RR$-module $\RR[\Omega]$ -- freely spanned by $\Omega$ -- carries canonically the structure of a discrete left $\RG$-module which is then defined to be a {\it discrete permutation $\RG$-module}. More precisely, as $\RR[G]$-modules, 
\begin{equation}
\RR[\Omega]\cong\coprod_{\omega\in\ca R}\RR[G/G_\omega],
\end{equation}
where $\ca R$ is a set of representatives of the $G$-orbits in $\Omega$. We call  {\it proper} those discrete permutation $\RG$-modules with compact stabilisers. A left resolution
$$\xymatrix{\cdots\ar[r]&P_n\ar[r]&P_{n-1}\ar[r]&\cdots\ar[r]&P_0\ar[r]&\RR\ar[r]&0}$$
of $\RR$ in $\RGdis$ will be called a ({\it proper}) {\it  discrete permutation resolution} of $\RR$ if each module $P_n$ is a (proper) discrete permutation $\RR[G]$-module.

\medskip

Let $\RR=\Q$. In this case, every proper discrete permutation $\QG$-module $\Q[\Omega]$ is a projective object in $\QGdis$. More precisely, one has the following useful characterisation.
\begin{fact}[\protect{\cite[Corollary~3.3]{it:ratdiscoh}}] 
Let $G$ be a t.d.l.c. group. A discrete $\QG$-module $M$ is projective if and only if $M$ is a direct summand of a proper discrete permutation $\QG$-module. 
\end{fact}
Therefore, by van Dantzig's theorem, one deduces that $\QGdis$ has also enough projectives. 
Note, though, that the same is not true in general for proper discrete permutation $\RG$-modules: the divisibility of $\Q$ is crucial to the argument. Moreover, in the special case of $G$ being a profinite group, one has that every discrete $\QG$-module is both injective and projective by \cite[Proposition 3.1]{it:ratdiscoh}.

\subsection{Limits and Colimits}

The abelian category $\RGdis$ has all small colimits: $\RGdis$ is a full subcategory of $\RGmod$ closed under taking colimits and $\RGmod$ has all small colimits.
On the other hand, limits in $\RGdis$ are not created in $\RGmod$ in general. For example, let $G=\Q_p$ be the field of $p$-adic rationals for some prime $p$ and $\{U_n=p^n\Z_p\mid n\in\N\}$ the nested family of open balls centred at $0$. The product
$\prod_{n\in\N} \RR[G/U_n]$ is not discrete: the element $(U_n)_{n\in\N}$ has stabiliser equal to $\bigcap_n U_n=1$, which is not open in $G$.

However, for arbitrary t.d.l.c. groups, $\RGdis$ does always admit all limits. To this end,  for $A\in\ob(\RGdis)$, let
\begin{equation}
\dd(A)=\{a\in A\mid \stab_G(a)\ \text{open in}\ G\},
\end{equation}
be the largest discrete $\RG$-submodule of $A$. The ``discretising" functor
\begin{equation}
\dd\colon\RGmod\to\RGdis
\end{equation}
 is a covariant additive left exact functor, which is the right-adjoint of the forgetful functor $\ff\colon\RGdis\to\RGmod$; see \cite[\S~2.2]{it:ratdiscoh} for the details. This adjunction makes $\RGdis$ a coreflective subcategory of $\RGmod$, so it has small limits because $\RGmod$ does: explicitly, limits in $\RGdis$ can be constructed by applying $\dd$ to limits in $\RGmod$, and we will write $\dd \lim$ for such limits to avoid confusion; $\lim$ will represent limits in $\RGmod$.
\begin{fact}\label{fact:lim} Let $G$ be a t.d.l.c. group. Then products in $\QGdis$ are exact.
\end{fact}
\begin{proof}
This holds in any abelian category with enough projectives.
\end{proof}


\subsection{The rational discrete standard $G$-bimodule}

In the context of discrete $\QG$-modules, there is an important discrete $\QG$-bimodules which can play some of the role of the group algebra: the so-called {\it rational discrete standard $G$-bimodule} $\biB(G)$. In the case that $G$ is unimodular, $\biB(G)$ turns out to be isomorphic but not canonically isomorphic to the $\Q$-vector space $\ca C_c(G,\Q)$ of continuous functions from $G$ to $\Q$  with compact support; see \cite{it:ratdiscoh} for the results. Here we recall the construction of $\biB(G)$, which will play a crucial role throughout this paper.

\medskip

Let $\CO(G)=\{U\leq G\mid U\ \text{compact and open}\}$.
For $U,V\in\CO(G)$, $V\subseteq U$, one has  an injective map 
\begin{equation}
\label{eq:biB1}
\eta_{U,V}\colon\Q[G/U]\longrightarrow\Q[G/V],\ \ \ 
\eta_{U,V}(x\,U)=\frac{1}{|U:V|}\sum_{r\in\caR} x\, r\, V,\ \ \ x\in G,
\end{equation}
of discrete $\QG$-modules.
By construction, one has that $\eta_{U,W}=\eta_{V,W}\circ\eta_{U,V}$ whenever $W\in\CO(G)$ with $W\subseteq V\subseteq U$. Let 
\begin{equation}
\biB(G)=\varinjlim_{U\in\CO(G)} (\Q[G/U],\eta_{U,V}).
\end{equation}
Then, by definition, $\biB(G)$ is a discrete left $\QG$-module and 
the assignment
\begin{equation}
\label{eq:biB2}
(xU)\cdot g=xg\,U^g, \ \ g,x\in G,\ \ U\in\CO(G)\ \text{and}\ U^g=g^{-1} U g,
\end{equation}
defines a right $\QG$-module structure on $\biB(G)$ that commutes with the left structure, i.e., $\biB(G)$ is a discrete $\QG$-bimodule, which is then called {\it rational discrete standard $G$-bimodule}.

\begin{prop}
\label{prop:bimodulesummand}
Let $G$ be a t.d.l.c. group and $H$ an open subgroup. Then as $\QH$-bimodules $\biB(H)$ is (isomorphic to) a direct summand of $\biB(G)$.
\end{prop}
\begin{proof}
We may define $\biB(G)$ as a direct limit $\varinjlim \Q[G/U]$ where $U$ ranges over the compact open subgroups contained in $H$, because this is cofinal in the poset of all compact open subgroups. For each compact open subgroup $U$ of $H$, as left $\QH$-modules $\Q[H/U]$ is a direct summand of $\Q[G/U]$, and this decomposition into direct summands $\Q[H/U] \oplus \Q[(G - H)/U]$ is easily seen to be compatible with \eqref{eq:biB1} and \eqref{eq:biB2} with respect to $H$, so as $\QH$-bimodules the direct limit $\varinjlim_{U \in \CO(H)} \Q[H/U] = \biB(H)$ is a direct summand of $\biB(G)$.
\end{proof}

\begin{prop}\label{prop:biG}
Let $G$ be a t.d.l.c. group. For every  closed normal subgroup $N$ of $G$, one has an isomorphism
$$\Q\otimes_N \biB(G)\cong\biB(G/N),$$
of $\Q[G/N]$-bimodules. 
\end{prop}
\begin{proof} Let $N$ be a closed normal subgroup of $G$. Firstly, note that the set $\CO_N(G/N)=\{UN/N\mid U\in\CO(G)\}$
is a local basis at $N$ in $G/N$ consisting of compact open subgroups, i.e., $\CO_N(G/N)$  is cofinal in $\CO(G/N)$ . Thus
\begin{align*}
\Q\otimes_N \biB(G) &\cong \Q\otimes_N\big( \varinjlim_{U\in\ca \CO(G)} \Q[G/U]\big)\quad\text{(by definition)} \\
&\cong \varinjlim_{U\in\CO(G)} \Q[N\backslash G/U]\quad\text{(since $\Q\otimes_N\argu$ commutes with direct limits)}\\
& \cong \varinjlim_{U\in\CO(G)} \Q[G/UN]\quad\text{ (for $N\backslash G/U= G/UN$ since $N$ is normal})\\
& \cong \varinjlim_{U\in\CO(G)} \Q[\frac{G/N}{UN/N}]\\
&\cong \biB(G/N)\quad\text{(for $\CO_N(G/N)$  is cofinal in $\CO(G/N)$)}.
\end{align*}
\end{proof}
We denote $\Q\otimes_N\biB(G)$ by $\biB(G)_N$.
\subsection{Flat discrete $\RG$-modules}\label{ss:flat}

A discrete left $\RG$-module $M$ is said to be {\it flat} if the functor
\begin{equation}
\argu\otimes_G M\colon\RGdis\to\Rmod,
\end{equation}
is exact. 

\medskip

Let $\RR=\Q$. Since every short exact sequence of discrete $\QG$-modules over a profinite group splits (cf. the end of \S~\ref{ss:perm}) and $\argu\otimes_G\argu$ preserves splittings in each argument, it is not difficult to prove that every proper discrete permutation $\QG$-module is flat. More generally, one has the following.
\begin{fact}\label{fact:flat} Let $G$ be a t.d.l.c. group.
\begin{enumerate}
\item Every projective discrete $\QG$-module is flat.
\item $\biB(G)$ is flat in $\QGdis$.
\end{enumerate}
\end{fact}
Moreover, one has the following characterisation of flat discrete $\QG$-modules that resembles Lazard's theorem for abstract modules.
\begin{thm}[Lazard-type Theorem]\label{thm:lazard} Let $G$ be a t.d.l.c. group and $M\in\textup{ob}(\QGdis)$. Thus the following are equivalent:
\begin{enumerate}[$(a)$]
\item the discrete $\QG$-module $M$ is flat;
\item for every morphism $f\colon P_1\to P_2$ of finitely generated proper discrete permutation $\QG$-modules and every morphism $\phi_2\colon P_2\to M$ such that $\phi_2\circ f=0$, there exists a morphism $f'\colon P_2\to P_3$ of finitely generated proper discrete permutation $\QG$-modules and a morphism $\phi_3\colon P_3\to M$ such that $f'\circ\ f=0$ and $\phi_2=\phi_3\circ f'$. In other words, the following lift occurs
\begin{equation}
\xymatrix{
                       &                      &P_3\ar[dr]^{\phi_3}                     &\\
P_1\ar[r]^{f}& P_2\ar@{-->}[ur]^{f'}\ar[dr]\ar[rr]^{\qquad\phi_2}&                                            &M\\
                       &                      &\textup{coker}{(f)}\ar[ur]\ar@{-->}[uu]          &
                       ,}
\end{equation}
i.e., every map $\alpha\colon \textup{coker}{(\alpha)}\to M$ factors through some $\phi'$.
\item $M$ is a direct limit of finitely generated proper discrete permutation $\QG$-modules.
\end{enumerate}
\end{thm}
\begin{proof}[Proof of Theorem~\ref{thm:lazard}] $(a)\Rightarrow(b)$ Let $f\colon P_1\to P_2$ and $\phi_2\colon P_2\to M$ as in $(b)$. Consider the rational discrete standard $G$-bimodule $\biB(G)$ and let
\begin{equation}
f^*=\argu\circ f\colon \Hom_G(P_2,\biB(G))\to\Hom_G(P_1,\biB(G))
\end{equation}
denote the map - induced by $f$ - of discrete right $\QG$-modules (where the right action on $\Hom_G(\argu,\biB(G))$ is given by the right module structure of $\biB(G)$) and denote by $K$ the kernel of $f^*$. Notice that $\textup{Hom}_G(P,\biB(G))$ is a finitely generated discrete right $\QG$-module with compact stabiliser whenever $P$ is a left one. Since $M$ is flat, one has the following commutative diagram
\begin{equation}
\xymatrix{
0\ar[r]&K\otimes_G M\ar[r]\ar[dr]& \Hom_G(P_2,\biB(G))\otimes_G M\ar[r]_{f^*\otimes\textup{id}_M}\ar[d]^{\wr}& \textup{Hom}_G(P_1,\biB(G))\otimes_G M\ar[d]^{\wr}\\
         &                                       & \Hom_G(P_2,M) \ar[r]                                                    & \Hom_G(P_1,M)
}
\end{equation}
in $\Qmod$ (cf. \cite[Prop.~4.6]{it:ratdiscoh}). Since $\phi_2\circ f=0$, $\phi_2$ can regarded as an element of $K\otimes_G M$, i.e.,
\begin{equation}
\phi_2\sim \sum_{i\in I} k_i\otimes_G m_i,
\end{equation}
for some $k_i\in K,\ m_i\in M$ and $I$ finite. Now let $P_3^{\times}$ be a finitely generated discrete right permutation $\QG$-module with compact stabilisers mapping onto the $G$-submodule of $K$ generated by $\{k_i\mid i\in I\}$. Namely,
\begin{equation}
P_3^{\times}=\coprod_{i\in I} \Q[U_i\backslash G]\to \Hom_G(P_2,\biB(G)),\quad U_i\mapsto k_i\ \forall i\in I
\end{equation}
where $U_i\subset\stab_G(k_i)$ for each $i\in I$. Dualizing yields a map
\begin{equation}
f'\colon P_3=\coprod_{i\in I}\Q[G/U_i]\to P_2,
\end{equation}
by \cite[\S~4.4]{it:ratdiscoh}, such that $f'\circ f=0$. Finally, $\phi_3\colon P_3\to M$ is defined by $U_i\mapsto m_i$.

\noindent $(b)\Rightarrow(c)$ We may write $M$ as a direct limit of finitely presented modules in the usual way; see \cite[Lemma 2]{Shannon}. Now (b) is telling us precisely that every map from a finitely presented module to $M$ factors through a map from a finitely generated proper discrete permutation module to $M$, so (c) follows by the same argument as \cite[Lemma 2]{Shannon}.
\noindent$(c)\Rightarrow(a)$ Since direct limits of flat modules are flat, Fact~\ref{fact:flat} concludes the proof.
\end{proof}

\subsection{Rational discrete homology}

Following \cite{it:ratdiscoh}, for every discrete right $\QG$-module $B$ we denote by
\begin{equation}\label{def:dTor}
\dTor^G_k(B,\argu)\colon\QGdis\to\Qmod,\quad k\geq0,
\end{equation}
the left derived functors of the right exact functor $B\otimes_G\argu\colon\QGdis\to\Qmod$. Clearly, one has $\dTor^G_0(B,A)=B\otimes_G A$ for any discrete left $\QG$-module $A$. The $\dTor$-spaces can be computed using flat resolutions of either the first or the second argument but one may also use projective resolutions for this purpose. The {\it rational discrete homology} of $G$ is defined by
\begin{equation}\label{def:homology}
\dH_k(G,\argu)=\dTor^G_k(\Q,\argu),\quad k\geq0,
\end{equation}
where $\Q$ is the trivial right $\QG$-module. As a direct consequence of the definition one has the following properties:
\begin{enumerate}
\item If $P$ is projective, then $\dTor^G_k(C,P)=0$ for all $k\geq1$.
\item If $G$ is compact, then $\dH_k(G,\argu)=0$ for all $k\geq1$.
\end{enumerate}

Similarly, we define $\dd\Ext$-functors as the right derived functors of $\Hom$, these may be calculated by taking a projective resolution of the first argument or an injective resolution of the second, we define group cohomology by $\dH^k(G,-) = \dd\Ext_G^k(\Q,-)$, and this is trivial for $G$ compact and $k \geq 1$.

\subsection{Module structures on $\dTor$ and $\dExt$}

As for abstract groups, when $N$ is a (closed) normal subgroup of $G$, and $B$ is a discrete right $\QG$-module, we may think of $\dTor^N_k(B,-)$ as a functor $\QGdis \to _{\Q[G/N]}\mathbf{mod}$. The $G/N$-action is induced by the action on the tensor product: if $A \in \QGdis$, $G$ acts on $B \otimes_N A$ by 
$$g \cdot (b \otimes a) = bg^{-1} \otimes ga,$$ which is trivial for $g \in N$.  In fact $\dTor^N_k(B,-)$ is  a functor from $\QGdis$ to $_{\Q[G/N]}\mathbf{dis}$. To see it we just need to check this for tensor products. First consider an element of $B \otimes_{N} A$ of the form $b \otimes a$. If $a$ is stabilised by some open $U$, and $b$ by some open $V$ then, for $g$ in  $U \cap V$, $g(b \otimes a) = bg^{-1} \otimes ga = b \otimes a$ as required. Therefore, for a general element of $B \otimes_{N} A$, the stabiliser contains a finite intersection of such subgroups, so it is open.

\medskip

On the other hand, for $A \in \QGdis$, we may think of $\Hom_{\Q[N]}(A,-)$ as a functor $\QGdis \to _{\Q[G/N]}\mathbf{mod}$ with the $G/N$-action given by 
$$g \cdot f (a)= gf(g^{-1}a), \quad \forall a\in A,$$ 
which is trivial for $g \in N$. Suppose $N=1$ and $U$ is a compact open subgroup of $G$ whose normal core is not open, e.g., a compact open subgroup in a (topologically) simple t.d.l.c. group. Then consider $\Hom_\Q(\bigoplus_{G/U} \Q, \Q[G/U]) = \prod_{G/U} \Q[G/U]$, where $\bigoplus_{G/U} \Q$ is given the trivial action. The $G$-action on $\prod_{G/U} \Q[G/U]$ is the diagonal one, acting by left-multiplication on each copy of $G/U$. But this is not a discrete $G$-module because the stabiliser in $G$ of the element $(gU)_{gU \in G/U}$ is the normal core of $U$.  Neretin's group of almost automorphisms of a regular tree is an example of a such a group; see \cite{Kap}. 

Therefore, we can define an `internal $\Hom$' functor 
\begin{equation}
\dd \circ \Hom_{\Q[N]}(A,-): \QGdis \to _{\Q[G/N]}\mathbf{dis},
\end{equation}
 which satisfies the usual adjunction with tensor products and is left exact. We write $\dd\Hom$ for this. Indeed, it is left exact because $\dd$ and $\Hom$ are; for $A,C \in \QGdis$, and $B \in {}_{\Q[G/N]}\mathbf{dis}$, we have
\begin{align*}
\Hom_{\Q[G/N]}(B, \dd\Hom_{\Q[N]}(A,C)) &= \Hom_{\Q[G/N]}(B, \Hom_{\Q[N]}(A,C)) \\
&= \Hom_{\Q[G]}(B \otimes_\Q A, C)),
\end{align*}
where the action on $B \otimes_\Q A$ is the diagonal one.

\medskip

This difference must be borne in mind whenever we want to define Lyndon-Hochschild-Serre-type spectral sequences, and care should always be taken over the distinction here between $\Hom$ and $\dd\Hom$. For example, since $\Hom$ is balanced, its derived functors may be calculated by taking a projective resolution of the first variable, an injective resolution of the second variable, or both; the same is not true of $\dd \Hom$.  In the context of a normal closed subgroup $N$ of a t.d.l.c. group $G$, we will take group cohomology $H^\bullet(N,-)$ to be the derived functors of $-^N$: this avoids any ambiguity because applying $-^N$ to a discrete $\QG$-module gives a discrete $\Q[G/N]$-module automatically. Observe, for $A \in \QGdis$, that:
\begin{enumerate}[(i)]
\item $H^n(N,A)$ can be calculated as $\dd\Ext_{\Q[N]}^n(\Q,A)$, even though $\dd\Ext_{\Q[N]}^n(-,-)$ is not in general a functor $\QGdis^{op} \times \QGdis \to _{\Q[G/N]}\mathbf{mod}$;
\item $H^n(N,A)$ can be calculated as the homology of $\dd\Hom_{\Q[N]}(\Q,I)$ where $I$ is an injective resolution of $A$, but not as the homology of $\dd\Hom_{\Q[N]}(P,A)$ for $P$ a projective resolution of $\Q$.
\end{enumerate}

\subsection{Topological modules}\label{ss:topmod}

As a concrete motivation for introducing this more general category of modules: we will want to study when discrete $\RG$-modules have resolutions by finitely generated proper discrete permutation modules, but when they are not projective (e.g. $\RR=\Z$), studying these resolutions has some extra obstacles. It turns out that the key to making things work is Corollary \ref{FPSchanuel}; the proof uses Theorem \ref{fpnkpn}, which in turn requires that we use topological $\RG$-modules.

\medskip 

In this paper, we consider algebraic objects in the category of $k$-spaces, that is, spaces which are compactly generated and weakly Hausdorff. For background on such spaces, see \cite{Lamartin}.

We can consider the categories of group objects, module objects, etc. in this category, and call such objects $k$-groups, $k$-modules, etc. See \cite[Section 1, Section 7]{Ged} for background on these categories. We summarise the details we will need.

First, all locally compact Hausdorff spaces are $k$-spaces, and so all t.d.l.c. groups are automatically $k$-groups.

The group ring $\RG$ can be given a topology making it a $k$-$R$-algebra satisfying the usual universal property: if $S$ is a $k$-$R$-algebra, every continuous group homomorphism from $G$ to the group of units of $S$ (with the subspace topology) factorises uniquely through $G \to \RG \to S$.

The category $\RGtop$ of $k$-$\RG$-modules is well-behaved: it is quasi-abelian, and has several interesting exact structures that make it into a left exact category (in the sense of \cite{Ged}; see there for details) with enough projectives. In fact this left exact structure is even exact, by \cite[Proposition 4.3]{BC}. Moreover there are well-behaved $\Hom_{\RG}(\argu,\argu)$ and $\argu\otimes_{\RG}\argu$ functors satisfying the usual form of adjunction.

The relevant left exact structure here is the one induced by taking the class of projectives to be summands of free modules on disjoint unions of compact Hausdorff spaces, which we call the {\it compact Hausdorff structure}. This helpfully makes the group algebra $\RG$ a projective $R$-module when $G$ is t.d.l.c., by van Dantzig's theorem. Indeed, we can say more:

\begin{lem}
\label{GprojH}
For $G$ a t.d.l.c. group and $H$ a (closed) subgroup, $\RG$ is projective as a $k$-$\RH$-module, in the compact Hausdorff structure.
\end{lem}
\begin{proof}
By \cite[Lemma 2.3]{AHV}, the quotient map $G \to H \backslash G$ has a continuous section. So as an $H$-space, $G$ is homeomorphic to $H \times H \backslash G$, where $H$ acts by left-multiplication on the first factor. Thus $\RG$ is the free $\RH$-module on the space $H \backslash G$ with the quotient topology, which is a disjoint union of compact Hausdorff spaces since it is homeomorphic to a subspace of $G$.
\end{proof}

We will write $\xrightarrow{CH}$ for the projective-split maps (that is, maps which have the right lifting property for projective objects) in this category, and call such maps {\em CH-split.}

Finally, note that $\RGdis$ is an abelian subcategory of $\RGtop$, and that the restriction of the compact Hausdorff structure of $\RGtop$ to $\RGdis$ gives the usual abelian structure.

\section{$G$-spaces with discrete actions}
\subsection{Discrete $G$-CW-complexes}\label{ss:CW}
Following \cite{it:ratdiscoh}, let $\eu F$ be a non-empty set of open subgroups of $G$ satisfying
\begin{itemize}
\item[(F1)] for $A\in\eu F$ and $g\in G$ one has ${}^gA=gAg^{-1}\in\eu F$;
\item[(F2)] for $A,B\in\eu F$ one has $A\cap B\in\eu F$.
\end{itemize}
A non-trivial topological space $X$ together with a continuous left $G$-action $\argu\cdot\argu\colon G\times X\to X$ is called a left {\it $G$-space}. Moreover a $G$-space is said to be {\it $\eu F$-discrete}, if $\stab_G(x)\in\eu F$ for all $x \in X$. We shall call a $G$-space $X$ {\it discrete}, whenever the family $\eu F$ is clear. If the family $\eu F$ is contained in $\CO(G)$, we say that the $G$-space is {\it proper}. A map of $G$-spaces is a continuous map $f\colon X\to Y$ of left $G$-spaces which commutes with the $G$-action.
\begin{defn}
A non-empty $G$-space $X$ together with an increasing filtration $(X_n)_{n\geq0}$ of closed subspaces $X_n \subseteq X$ is called a {\it $(G,\eu F)$-CW-complex}, if
\begin{itemize}
\item[(D1)] $X = \bigcup_{n\geq 0}X_n$;
\item[(D2)] $X_0$ is a $\eu F$-discrete subspace of $X$;
\item[(D3)] for $n \geq 1$ there exist a $\eu F$-discrete space $\Lambda_n$, $G$-maps $f\colon S^{n-1} \times \Lambda_n \to X_{n-1}$ and
$\hat f\colon B^n \times \Lambda_n\to X_n$ such that
$$\xymatrix{
S^{n-1} \times \Lambda_n\ar[r]^f\ar[d]&X_{n-1}\ar[d]\\
B^n \times \Lambda_n\ar[r]^{\hat f}&X_n}$$
is a push-out diagram, where $S^{n-1}$ denotes the unit sphere and $B^n$ the unit ball in
euclidean $n$-space (with trivial $G$-action);
\item[(D4)] a subspace $Y \subset X$ is closed if and only if $Y \cap X_n$ is closed for all $n \geq 0$.
\end{itemize}
\end{defn}
\begin{rem}
All the spaces $\Lambda_n$ comes equipped with the discrete topology since $\eu F$ is a family of open subgroups of $G$.
Therefore a discrete $G$-CW-complex is, forgetting the group action, a CW-complex.
\end{rem}
\begin{fact} Let $G$ be a t.d.l.c. group and $X$ a discrete $G$-CW-complex. Then
\begin{enumerate}
\item the action of $G$ on $X$ is continuous;
\item the action of $G$ on $X$ is by cell-permuting homeomorphisms;
\item an element in $G$ fixing a cell $\sigma$ of $X$ setwise fixes $\sigma$ pointwise.
\end{enumerate}
\end{fact}
We shall refer to $X$ simply as {\it discrete $G$-CW-complex} when there is no need to specify the family $\eu F$. Moreover, $X$ will be called {\it proper} if $\eu F\subseteq\CO(G)$, i.e., the cell-stabilisers are compact and open.
\begin{rem}
In \cite{sauert}, a $(G,\CO(G))$-CW-complex $X$ is called {\it proper smooth $G$-$\textup{CW}$-complex}. Moreover, if $X$ is also contractible, then $X$ is called a {\it topological model} of $G$.
\end{rem}
\begin{rem} Alternatively, one could allow actions with inversions and talk about {\it discrete $G$-complexes}: this is the approach taken in \cite{brown:presentations}, for example. In such a case, the associated cellular chain complex is no longer formed by  discrete permutation modules because some kind of ``twisting" is involved. Nevertheless, many of the upcoming results can be stated for discrete $G$-complexes as well.
\end{rem}
 The {\it dimension} of $X$ is defined by
\begin{equation}
\dim(X)=\min(\{k\in \N_0 \mid X_k =X\}\cup\{\infty\}),
\end{equation}
and $X$ is said to be {\it of type} $F_n$, $n \in \N_0 \cup \{\infty\}$, if $G$ has finitely many orbits on the $k$-skeleton $X_k$ of $X$ for $0 \leq k \leq n$.
\begin{ex} The topological realisation of a simplicial complex acted on by $G$ with open stabilisers is an example of discrete $G$-CW-complex. For example, let $G$ be a compactly generated t.d.l.c. group. Recall that every locally finite connected graph equipped with a transitive $G$-action with compact open stabilisers is called a {\em Cayley-Abels graph} of $G$. For every compact open subgroup $U\leq G$ there exists a Cayley-Abels graph whose vertices are taken to be the cosets $G/U$. The topological realisation of a Cayley-Abels graph of $G$ is a proper discrete $G$-CW-complex.
\end{ex}
\subsection{Cellular homology with coefficients in $\RR$}
Since the ordinary theory of $CW$-complexes is easily extended to the equivariant setting with discrete actions, one can associate to any discrete $G$-CW-complex $X$ the cellular chain complex $C_\bullet(X,\RR)=C_\bullet(X)\otimes_\Z\RR$ with coefficients in $\RR$.
\begin{fact}\label{fact:cchain}
Let $G$ be a t.d.l.c. group and $X$ a contractible (proper) discrete $G$-CW-complex. Then the augmented cellular chain complex
$$\xymatrix{\cdots\ar[r]&C_n(X,\RR)\ar[r]^\delta&C_{n-1}(X,\RR)\ar[r]^-\delta&\cdots\ar[r]&C_0(X,\RR)\ar[r]^-\epsilon&\RR\ar[r]&0}$$
is a (proper) discrete permutation resolution of $\RR$ in $\RGdis$, where $\epsilon$ denotes the augmentation map $\sigma\mapsto 1$ (for all $0$-cells $\sigma$ in $X$). For $\RR=\Q$ and $X$ proper, one obtains a projective resolution of $\Q$ in $\QGdis$.
\end{fact}
By construction, the homology of $(C_\bullet(X,\RR),\delta_\bullet)$ is the cellular homology of $X$ with coefficients in $\RR$.
Moreover, one can also consider the {\it reduced cellular chain complex of $X$} with coefficients in $\RR$ defined by
\begin{equation}
\tilde{C}_p(X,\RR)=
\begin{cases}
C_p(X,\RR)& p>0,\\
\ker(\epsilon)& p=0.
\end{cases}
\ \text{and}\quad \tilde\delta_p=\delta_p,\ \forall p\geq0.
\end{equation}
and its homology is thus the reduced homology of $X$ with $\RR$-coefficients and satisfies the following well-known property.
\begin{prop}[\protect{\cite[Ch. 4, Lemma 1]{spanier}}]\label{prop:redhom} Let $G$ be a t.d.l.c. group and $X$ a discrete $G$-CW-complex; then
\begin{equation}
\mathrm H_p(X,\RR)=
\begin{cases}
\tilde{\mathrm H}_p(X,\RR)& p>0,\\
\tilde{\mathrm H}_0(X,\RR)\oplus \RR& p=0.
\end{cases}
\end{equation}
\end{prop}
\subsection{Rational discrete equivariant homology}
Let $X$ be a discrete $G$-CW-complex $X$. For every discrete left $RG$-module $M$, set
\begin{equation}\label{eq:ccc}
C_\bullet(X,M)=C_\bullet(X,\RR)\otimes_\RR M\quad\text{and}\quad \delta_\bullet=\delta_\bullet\otimes_\RR\mathrm{id}_M,
\end{equation}
and equip $C_\bullet(X,M)$ with the diagonal $G$-action. Thus \eqref{eq:ccc} is a chain complex in $\RGdis$ which is called {\it cellular chain complex of $X$ with coefficients in $M$}. For $R=\Q$, we will refer to
\begin{equation}
\dH^G_k(X;M):=\mathrm{H}_k(P\otimes_{G} C(X,M)),\quad k\geq0,
\end{equation}
as {\it rational discrete equivariant homology of $(G,X)$ with coefficients in $M$}, where $P\to\Q$ is a projective resolution of $\Q$ of discrete right $\QG$-modules.

\medskip

By definition, $\dH^G_k(X;M)$ is the homology of the total complex associated to the double complex $P_q\otimes_G C_p(X,M)$ (or, alternatively, to $P_p\otimes_G C_q(X,M)$). Therefore the standard theory yields two spectral sequences for computing $\dH^G_k(X,M)$, where the $E^2$-term in each case is the horizontal homology of the vertical homology (cf. \cite[Ch. VII]{brown}). Namely, one has
\begin{equation}\label{eq:ss1}
E^1_{p,q}=\dH_q(G;C_p(X,M))\Rightarrow \dH^G_{p+q}(X;M),\quad p,q\geq0,
\end{equation}
and
\begin{equation}\label{eq:ss2}
E^2_{p,q}=\dH_p(G;H_q(X;M))\Rightarrow \dH^G_{p+q}(X;M),\quad p,q\geq0.
\end{equation}
To analyse  $P_q\otimes_G C_p(X,M)$ further, let $M^\times$ denotes the module $M\in\ob(\QGdis)$ regarded as discrete right $\QG$-module. There is a natural isomorphism\footnote{One can define the isomorphism directly or deduce it from $\argu\otimes_G\argu\cong(\argu\otimes_\Q\argu)_G.$} 
\begin{equation}
(A\otimes_R B^\times)\otimes_G C\cong (A\otimes_R C^\times)\otimes_G B,\quad\forall A\in \textup{ob}(\disQG),BC\in\textup{ob}(\QGdis),
\end{equation}
where $G$ acts on each tensor product $\argu\otimes_R\argu$ by the diagonal action. Therefore in each dimension $p+q$ one has
\begin{equation}
P_q\otimes_G C_p(X,M)\cong(P_q\otimes_\Q C_p(X,\Q)^\times)\otimes_G M.
\end{equation}
So if we take the vertical homology (fixing $p$ and taking the homology with respect to $q$), we get
\begin{equation}\label{eq:ss3}
E^1_{p,q}=\dTor^G_q(C_p(X,\Q)^\times;M)\Rightarrow \dH^G_{p+q}(X;M),\quad p,q\geq0,
\end{equation}
which gives \eqref{eq:ss1} back in different terms.
\section{Type $\F_n$, Type $\FP_n$ and Type $\KP_n$}\label{s:comparison}
\subsection{T.d.l.c. groups of type $\F_n$}\label{ss:typeF}
A t.d.l.c. group G is said to be {\it of type $\F_n$} ($n\in\N\cup\{\infty\}$) if there exists a contractible proper discrete $G$-CW-complex $X$ of type $\F_n$, i.e., the $n$-skeleton of $X$ has finitely many $G$-orbits. Moreover, the t.d.l.c. group $G$ is {\it of type $\F$} if  there is a contractible proper discrete $G$-CW-complex $X$ of type $F_\infty$ with $\dim(X)<\infty$.

\begin{ex}
\label{ex:fp}
\begin{enumerate}[(i)]
\item Neretin's group of almost automorphisms of a regular tree has type $\F_\infty$, by \cite{sauert}.
\item Hyperbolic t.d.l.c. groups have type $\F$. Indeed one can construct a contractible Rips' complex of finite dimension (cf. Fact~\ref{fact:rips}).
\item Simply-connected semi-simple algebraic groups defined over a non-discrete non-archimedean local field have type $\F$.
\item Certain Kac-Moody groups have type $\F$; see \cite{it:ratdiscoh} for details.
\end{enumerate}
\end{ex}

\subsection{Compactly presented t.d.l.c. groups} We start by recalling the notion of generalised presentation of a t.d.l.c. group $G$. For a detailed definition of a graph of profinite groups and its fundamental group the reader is referred to \cite[\S 5.5]{it:ratdiscoh}.

\medskip

A {\it generalised presentation} of $G$ is a graph of profinite groups $(\ca A,\Lambda)$ together with a continuous open epimorphism
\begin{equation}
\phi\colon \pi_1(\ca A,\Lambda)\to G
\end{equation}
such that $\phi\restriction_{\ca A_v}$ is injective for all $v\in\ca V(\Lambda)$.
\begin{defn}
A t.d.l.c. group $G$ is defined to be {\it compactly presented} if there exists a generalised presentation $((\ca A,\Lambda),\phi)$ of $G$ such that
\begin{itemize}
\item[(G1)] $\Lambda$ is a finite connected graph, and
\item[(G2)] $N$ is finitely generated as normal subgroup of $\pi_1(\ca A,\Lambda)$.
\end{itemize}
\end{defn}
A generalised presentation satisfying both $(G1)$ and $(G2)$ will be called {\it finite}. Notice that $G$ is compactly generated if and only if there exists a generalised presentation based on a finite connected graph $\Lambda$. Clearly, being compactly presented implies being compactly generated.
\begin{rem}
In the literature there is a widely used equivalent definition of compact presentability based on the notion of {\it compact presentation}: $G$ is compactly presented if, as an abstract group, it admits a presentation $\langle S\mid R\rangle$ where $S$ is a compact set of generators and $R$ is a set of relators of bounded length. The notion of generalised presentation may have already been implicit in \cite{abels}. Anyway it has independently been made explicit in \cite{it:ratdiscoh} and \cite[Corollary 8.A.17]{cdlh} (up to Bass-Serre theory).
\end{rem}
\begin{prop}
\label{prop:typeF2}
Let $G$ be a t.d.l.c. group. Then
\begin{enumerate}[(i)]
\item $G$ is compactly generated if and only if $G$ is of type $\F_1$;
\item $G$ is compactly presented if and only if $G$ is of type $\F_2$.
\end{enumerate}
\end{prop}
\begin{proof}
Being of type $\F_1$ (respectively, $F_2$) implies compact generation (respectively, presentation) by \cite[Prop. 2.5]{sauert}.

Suppose now that $G$ is compactly generated. Let $((\ca A,\Lambda),\phi)$ be a generalised presentation of $G$ based on a finite connected graph $\Lambda$. Let $\ca T$ denote the Bass-Serre tree of $\Pi=\pi_1(\ca A,\Lambda)$ and $N=\ker(\phi)$. Thus the quotient graph $\Gamma=\ca T\quot N$ is a Cayley-Abels graph of $G$, which is cocompact. Since $\ca T$ is the universal cover of $\Gamma$, $\pi_1(\Gamma) = N$ (after fixing some basepoint); the action of $G$ on $\Gamma$ induces conjugation on $N$.
Let $\{\omega_i\}_{i\in I}$ be a set of normal generators of $N$ in $G$, and identify these with the corresponding loops in $\Gamma$. When $G$ is compactly presented, choose this set to be finite. Now attach a $G$-orbit of $2$-cells to the $G$-orbit of each of these loops: this space is simply connected because the $\omega_i$ normally generated $N$. Then we can add higher cells to kill higher homotopy. By Whitehead, the resulting $G$-CW-complex is contractible.
\end{proof}
In a t.d.l.c. group $G$, a {\it group retract} is a closed subgroup $H$ of $G$ such that there exists a continuous epimorphism $\rho\colon G\to H$ whose restriction to $H$ is the identity map.
\begin{cor}\label{cor:retract}
Let $G$ be a compactly presented t.d.l.c. group and $H$ a group retract of $G$. Then $H$ is compactly presented.
\end{cor}
\begin{proof}
This holds by the same argument as \cite[Theorem 6]{Ratcliffe}.
\end{proof}
\begin{rem} The general version of Corollary~\ref{cor:retract} in the context of $\sigma$-compact locally compact groups has been proved in \cite[8.A.12]{cdlh} from the point of view of coarse geometry.
\end{rem}
The next proposition imitates the well-known fact that maps $\Hom_{Grp}(G,-)$ from finitely presented groups $G$ commute with filtered colimits.
\begin{prop}
\label{compprescolimit}
\begin{enumerate}[(i)]
\item Suppose $(H_n)$ is a sequence of t.d.l.c. groups with quotient maps $h_n: H_n \to H_{n+1}$, and that the abstract colimit $H$, together with the quotient topology, is Hausdorff, and hence a t.d.l.c. group. If $G$ is compactly presented, then any map $f: G \to H$ factors through some $G \to H_n$.
\item Every compactly generated t.d.l.c. group $G$ can be written as a colimit of such a sequence $(G_n)$, with every $G_n$ compactly presented.
\item If $G$ is compactly presented, and written as a colimit of such a sequence $(G_n)$ with every $G_n$ compactly generated, this sequence stabilises, i.e. there is some $n$ such that $G_n \cong G$.
\end{enumerate}
\end{prop}
\begin{proof}
\begin{enumerate}[(i)]
\item Write $k_n$ for the canonical map $H_n \to H$. Fix a compact open subgroup $U$ of $H_0$; the image $U'$ of this in $H$ is again a compact open subgroup (because $H$ has the quotient topology from $H_0$). Let $V' = f^{-1}(U')$ and pick a compact open subgroup $V$ of $V'$.

We start by showing that $f\vert_V$ factors through some $V \to H_n$. Indeed, let $W \leq U$ be the preimage in $U$ of $f(V)$, $K$ the kernel of $k_0\vert_W$, and $K_n$ the kernel of $(h_{n-1}h_{n-2}\ldots h_0)\vert_W$. By hypothesis, $\bigcup_n K_n = K$, so by standard techniques for profinite groups (resulting from the Baire category theorem) $K_m = K$ for some $m$. Write $g_n$ for the resulting map $V \to H_n$ for $n \geq m$.

Next, choose a  generating system $G = \langle V,S \rangle$ with $S$ a finite symmetric set disjoint from $V$. From this, construct a generalised presentation $\phi: \pi_1(A,\Lambda) \to G$ using the construction from the proof of \cite[Proposition 5.10]{it:ratdiscoh}. We claim there is some $n$ such that $f\phi$ factors through some $\pi_1(A,\Lambda) \to H_n$.

Write $g'_m$ for some choice of lift $S \to H_m$ of $f\vert_S$ such that if $s,s^{-1} \in S$, $g'_m(s^{-1})=g'_m(s)^{-1}$; write $g'_n$ for the composite $h_{n-1}h_{n-2}\ldots h_mg'_m$. To prove the claim, it remains to show for each $s \in S$ that there is some $n \geq m$ such that $g_n(sxs^{-1}) = g'_n(s)g_n(x)g'_n(s)^{-1}$ for all $x \in V \cap s^{-1}Vs$ -- then all the relations of \cite[(5.15)]{it:ratdiscoh} will be satisfied. For each $n \geq m$, write $V_n$ for the set of $x \in V \cap s^{-1}Vs$ for which this holds; note $V_n$ is closed in $V \cap s^{-1}Vs$ because $H_n$ is Hausdorff, and $\bigcup_{n \geq m} V_n = V \cap s^{-1}Vs$ by hypothesis. By the Baire category theorem, there is some $m' \geq n$ such that $V_{m'}$ contains an open set of $V \cap s^{-1}Vs$, so it contains a coset $vZ$ of some open subgroup $Z$ with $v \in V \cap s^{-1}Vs$. As $Z$ has finite index in $V \cap s^{-1}Vs$, a quick calculation shows there is some $m'' \geq m'$ such that $V_{m''} = V \cap s^{-1}Vs$.

Since $S$ is finite, and this holds for all $s \in S$, the claim follows. Then we are done by the same method as for finitely presented abstract groups, since by \cite[Proposition 5.10]{it:ratdiscoh} $\ker(\phi)$ is finitely generated.
\item If $G$ is compactly presented, take $G_n=G$ for all $n$; assume it is not. By the proof of \cite[Proposition 5.10]{it:ratdiscoh}, construct a compactly presented t.d.l.c. group $G_0$ with a quotient map $G_0 \to G$. The kernel $K$ of this map is discrete, and countable because $G_0$ is $\sigma$-compact. Now enumerate the elements $(k_n)$ of $K$, and define $G_n = G_0/\overline{\llangle k_0,\ldots,k_n\rrangle}$ for $n \geq 1$: clearly these $G_n$, with the obvious maps between them, give the sequence we want.
\item The result follows in the same way as for finitely presented abstract groups, using similar techniques to (i); it is left as an exercise. 
\end{enumerate}
\end{proof}
Note the condition that $\colim_n G_n = G$ as abstract groups as well as t.d.l.c. groups -- otherwise it is easy enough to find counterexamples using profinite groups, which are always compactly presented. We will describe this situation by saying that $G$ is the abstract and topological colimit of the sequence.

\subsection{T.d.l.c. groups of type $\FP_n$ and $\KP_n$}
A discrete (left) $\RG$-module $M$ is said to be {\it finitely generated} if it contains a set of finitely many elements which is not contained in any proper submodule. A resolution $P\to M$ in $\RGdis$ is said to be {\it finitely generated} (or of {\it finite type}) if the discrete $\RG$-module $P_k$ is finitely generated in each dimension $k\geq0$. Recall that, by van Dantzig's theorem, every discrete $\RG$-module $M$ has a proper discrete permutation resolution. If the discrete $\RG$-module $M$ admits a finitely generated proper discrete permutation resolution then we say that $M$ has {\it type $\FP_\infty$ over $R$}. Note that we are not asking for  projective permutation modules. 

Moreover, if the module $M$ has a finitely generated partial proper discrete permutation resolution of length $n$, i.e., there exists $n\geq0$ such that $P_k$ is finitely generated for $0\leq k\leq n$, we say that $M$ has {\it type $\FP_n$ over $\RR$}. If $M$ has a finite type proper discrete permutation resolution of finite length, we say $M$ has {\it type $\FP$ over \RR}. 
E.g., $M$ is of type $\FP_0$ if and only if $M$ is finitely generated and that $M$ is of type $\FP_1$ if and only if $M$ is finitely presented. Note that the (partial) resolution $P\to M$ is not required to be projective; it will be so automatically for $\RR=\Q$.

\medskip

Now, given a $k$-$\RG$-module $A$, we say it is {\it compactly generated} if there is a CH-split map onto $A$ from a free module on a compact Hausdorff space. Consider projective resolutions of $A$ in the category $\RGtop$. Explicitly, this means an exact chain complex $$\cdots \to P_1 \xrightarrow{d_1} P_0 \xrightarrow{d_0} A \to 0$$ such that, for all $n$, the map $P_n \to \ker(d_{n-1})$ is CH-split. We say $A$ has {\it type $\KP_n$ over $\RR$}, $n \leq \infty$, if it has a projective resolution $P\to A$ in $\RGtop$ with $P_i$ compactly generated for $i \leq n$. Schanuel's lemma applies to these projective resolutions, so by standard methods we get:

\begin{lem} Let $G$ be a t.d.l.c. group and $A$ a $k$-$\RG$-module.
\begin{enumerate}[(i)]
\item If $A$ has type $\KP_n$ over $\RR$ and $$P_{n-1} \to P_{n-2} \to \cdots \to P_0 \to A \to 0$$ is any partial projective resolution of $A$ in $\RGtop$ with $P_i$ compactly generated for $i \leq n-1$, then $\ker(P_{n-1} \to P_{n-2})$ is compactly generated.
\item If $A$ has type $\KP_n$ over $\RR$ for all finite $n$, it has type $\KP_\infty$ over $\RR$.
\end{enumerate}
\end{lem}

We can also define type $\K_n$ by considering the contractible cofibrant objects with cocompact $n$-skeleta, in the compact Hausdorff model structure on $G$-$k$-spaces (see \cite{Ged}), by analogy with the abstract case.

Trivially, profinite groups have type $\K_\infty$ and $\KP_\infty$ (consider the bar resolution). By inducing from a compact open subgroup, finitely generated proper discrete $\RG$-permutation modules have type $\KP_\infty$.

\begin{prop}
\label{prop:KPn}
Let $0\to A' \xrightarrow{f} A \xrightarrow{g} A'' \to 0$ be a short exact sequence of $k$-$\RG$-modules. Then the following statements hold over $\RR$:
\begin{enumerate}[(a)]
\item If $A'$ has type $\KP_{n-1}$ and $A$ has type $\KP_n$, then $A''$ has type $\KP_n$;
\item If $A$ has type $\KP_{n-1}$ and $A''$ has type $\KP_n$, then $A'$ has type $\KP_{n-1}$;
\item If $A'$ and $A''$ have type $\KP_n$ then so does $A$.
\end{enumerate}
\end{prop}
\begin{proof}
\begin{enumerate}[(a)]
\item Take a type $\KP_{n-1}$ resolution $P'$ of $A'$ and a type $\KP_n$ resolution $P$ of $A$. There is a map $P' \to P$ extending $A' \to A$. The mapping cone of this is a type $\KP_n$ resolution of $A''$.
\item Fix a map $Q \stackrel{q}{\twoheadrightarrow} A$ with $Q$ compactly generated projective. By (a), $\ker(q)$ is of type $\KP_{n-2}$ and $\ker(gq)$ is of type $\KP_{n-1}$. By the snake lemma, $0 \to \ker(gq) \to \ker(q) \to A' \to 0$ is exact. By (a), $A'$ is of type $\KP_{n-1}$.
\item Use the Horseshoe lemma.
\end{enumerate}
\end{proof}

\begin{thm}
\label{fpnkpn}
Let $G$ be a t.d.l.c. group. A discrete $\RG$-module $M$ has type $\FP_n$ over $\RR$ if and only if it has type $\KP_n$ over $\RR$.
\end{thm}
\begin{proof}
First note that $M$ is finitely generated ($\FP_0$ over $\RR$) if and only if it is compactly generated ($\KP_0$ over $\RR$), because the image of any compact space in $M$ is finite. Indeed, given $f: R[G \times X] \xrightarrow{CH} M$ with $X$ compact, the restriction $G \times X \to M$ factors through a discrete quotient $G/H \times X' \to M$ with $H$ compact and open and $X'$ finite. Thus $f$ factorises through the induced map $f':R[G/H \times X'] \to M$, so $f'$ is CH-split. Conversely, any surjective map onto a discrete module is CH-split, so any finite choice $S$ of generators for $M$ gives a CH-split map $R[G \times S] \to M$.

Now we argue inductively, using Proposition \ref{prop:KPn}: we have already shown the base case holds.

Suppose $M$ has type $\KP_n$, so it is finitely generated. Take a finitely generated proper discrete permutation $\RG$-module $P$ and an epimorphism $P \to M$ with kernel $K$. $P$ has type $\KP_\infty$ so, by Proposition \ref{prop:KPn}, $K$ has type $\KP_{n-1}$; by hypothesis, $K$ has type $\FP_{n-1}$, so we can concatenate sequences to show $M$ has type $\FP_n$.

Conversely, suppose $M$ has type $\FP_n$. Take a length $n$ partial resolution of $M$ by finitely generated proper discrete $\RG$-permutation modules $$P_n \to \cdots \to P_1 \xrightarrow{d_1} P_0 \xrightarrow{d_0} M \to 0.$$  Then $\ker(d_0)$ has type $\FP_{n-1}$ so by hypothesis it has type $\KP_{n-1}$. $P_0$ has type $\KP_\infty$ so by Proposition \ref{prop:KPn} $M$ has type $\KP_n$.
\end{proof}

\begin{cor}
\label{FPSchanuel} Let $G$ be a t.d.l.c. group and $M$ a discrete $\RG$-module.
\begin{enumerate}[(i)]
\item If $M$ has type $\FP_n$ over $\RR$ and $$P_{n-1} \to P_{n-2} \to \cdots \to P_0 \to A \to 0$$ is a partial proper discrete permutation resolution of finite type in $\RGdis$, then $\ker(P_{n-1} \to P_{n-2})$ is finitely generated.
\item If $M$ has type $\FP_n$ over $\RR$ for all finite $n$, it has type $\FP_\infty$ over $\RR$.
\end{enumerate}
\end{cor}

\begin{cor}
\label{cor:FPn}
Let $0\to M' \xrightarrow{f} M \xrightarrow{g} M'' \to 0$ be a short exact sequence of discrete $\RG$-modules. Then the following statements hold over $\RR$:
\begin{enumerate}[(a)]
\item If $M'$ is of type $\FP_{n-1}$ and $M$ of type $\FP_n$, then $M''$ is of type $\FP_n$;
\item If $M$ is of type $\FP_{n-1}$ and $M''$ of type $\FP_n$, then $M'$ is of type $\FP_{n-1}$;
\item If $M'$ and $M''$ are of type $\FP_n$ then so is $M$.
\end{enumerate}
\end{cor}

The t.d.l.c. group $G$ is said to be {\it of type $\FP_n$} (respectively, $\FP$) over $R$, $n\in\N\cup\{\infty\}$, if the trivial $\RG$-module $\RR$ is of type $\FP_n$ (respectively, $\FP$). Since $\Q$ is flat over $\Z$, if $A$ has type $\FP_n$ or $\FP$ over $\Z$, $A \otimes_\Z \Q$ has the same type over $\Q$; in particular, if $G$ has type $\FP_n$ or $\FP$ over $\Z$, it has the same type over $\Q$.

\medskip
Clearly, a t.d.l.c. group $G$ of type $F_n$ is also of type $\FP_n$ ($n\in\N\cup\{\infty\}$) over any ring $R$, by Fact~\ref{fact:cchain}. By considering discrete groups, it is clear that $\FP_n$ does not imply $\F_n$ in the context of t.d.l.c. groups. The following result shows that it does hold for compactly presented t.d.l.c. groups.

\begin{prop}
\label{prop:FvsFP}
For compactly presented t.d.l.c. groups, being of type $\F_n$ over $\Z$ is equivalent to being of type $\FP_n$ over $\Z$.
\end{prop}
\begin{proof}
As for abstract groups, this is a consequence of the Hurewicz theorem and Corollary \ref{FPSchanuel}.
\end{proof}

\begin{rem}
This is the motivation for working with discrete permutation modules over $\Z$ rather than restricting to $\Q$-modules where we have enough projectives: the Hurewicz theorem allows stronger conclusions from knowing the $\Z$-homological type.
\end{rem}
Moreover, we can immediately deduce by Theorem~\ref{fpnkpn} that various t.d.l.c. groups have type $\KP_\infty$; see Example~\ref{ex:fp}. On the other hand, note that a t.d.l.c. group of type $\FP$ need not have type $\KP$: finite groups considered as discrete t.d.l.c. groups give counter-examples.
We can use Theorem \ref{fpnkpn} to obtain the following analogue of the invariance of type $\FP_n$ under commensurability for abstract groups.

Recall by Lemma \ref{GprojH} that quotient maps $f: G \to G/H$ of t.d.l.c. groups, where $H$ is a closed subgroup of $G$, are topologically split, that is, there is a continuous map of spaces $s: G/H \to G$ such that $fs$ is the identity on $G/H$.

\begin{lem}
\label{lem:cominvariant}
\begin{enumerate}[(i)]
\item If $G$ is a $k$-group and $H$ is a closed subgroup such that $G/H$ is compact and $G \to G/H$ is topologically split, $G$ has type $\KP_n$ over $\RR$ if and only if $H$ does.
\item For $G$ a t.d.l.c. group, $G$ has type $\FP_n$ over $\RR$ if and only if $H$ does, and $G$ has type $\F_n$ if and only if $H$ does.
\end{enumerate}
\end{lem}
\begin{proof}
$\RG$ is free as a $k$-$\RH$-module $\RR[H \times G/H]$. So a compact type partial projective resolution of length $n$ for $G$ is also one under restriction to $H$. If $H$ has type $\KP_n$ and hence (by induction) $G$ has type $\KP_{n-1}$, take a partial projective resolution of length $n-1$ for $G$: the kernel of the $(n-1)$th map is compactly generated as an $\RH$-module by Schanuel's lemma, so it is compactly generated as an $\RG$-module.

Thanks to Theorem \ref{fpnkpn}, the type $\FP_n$ part of (ii) follows immediately. The type $\F_n$ part follows because $G$ is compactly presented if and only if $H$ is by \cite[Corollary 8.A.5]{cdlh}. (`Cocompact' in the statement of \cite[Corollary 8.A.5]{cdlh} is equivalent to $G/H$ being compact by \cite[Lemma 2.C.9]{cdlh}.)
\end{proof}

In particular this holds when $H$ has finite index in $G$.

\begin{rem}
Part (ii) of this lemma can be also deduced from Corollary~\ref{cor:qi}.
\end{rem}

\begin{lem}
\label{lem:fpinduction}
If $H$ is a closed subgroup of a $k$-group $G$ and $A$ is a $k$-$R[H]$-module, $A$ has type $\KP_n$ over $\RR$ if and only if $\ind^G_HA = R[G] \otimes_H A$ has type $\KP_n$ over $\RR$.
\end{lem}
\begin{proof}
Induction is exact because $\RG$ is free as a $k$-$\RH$-module, so the `only if' part is clear. In the other direction, it suffices to prove the result when $n=0$; the rest follows by induction on $n$.

So suppose $\ind^G_HA$ is compactly generated, that is, we have a CH-split map $f: \RR[G \times X] \to \RG \otimes_{\RH} A$ with $X$ a compact Hausdorff space. We may identify $X$ with its image in $\ind^G_HA$: indeed, if $X$ is not homeomorphic to $f(X)$, the induced map $\RR[G \times f(X)] \to \ind^G_HA$ is CH-split too, so we replace $X$ with $f(X)$. Thanks to Lemma \ref{GprojH}, $\RG \otimes_{\RH} A \cong \RR[G/H] \otimes_{\RR} A$. Choose a continuous section $s: G/H \to G$, so that every `pure' element of $\RR[G/H] \otimes_R A$, of the form $x \otimes a$, can be written as $x' \otimes a'$ with $x' \in s(G/H)$. Recall from the construction of tensor products of $k$-modules in \cite[Section 7]{Ged} that, topologically, $\RR[G/H] \otimes_R A$ is the colimit of the sequence of closed subspaces $S_n$ consisting of elements that can be expressed as a sum of at most $n$ terms of the form $x \otimes a$, $x \in s(G/H), a \in A$. Hence any compact subspace of $\RR[G/H] \otimes_{\RR} A$ must be contained in some $S_N$; in particular this holds for $X$. Now define the compact subset $Y_n$ of $A$ by letting $y \in Y_n$ if there is some $x \in X$ with $x = \sum_{i=1}^n x_i \otimes a_i, x_i \in s(G/H), a_i \in A$, and $y=a_i$ for some $i$. Lastly, set $Y = \bigcup_{n=1}^N Y_n$, with the subspace topology from $A$, and this is once again compact.

It is easy to check that $Y$ generates $A$ as an abstract module. To see that $\RR[H \times Y] \to A$ is CH-split, observe that we have a commutative diagram of $\RH$-modules
\[\xymatrix{
\RR[H \times Y] \ar[r] \ar[d] & \RR[G \times Y] \ar[d] \\
A \ar[r] & \ind^G_HA. \\
}\]
From the construction of $Y$, we see that $f$ factors through $$\RR[G \times Y] = \ind^G_H\RR[H \times Y] \to \ind^G_HA,$$ so this map is CH-split. Finally, observe, from the construction, that for a compact subspace of $\ind^G_HA$ contained in $A$, a continuous section may be chosen whose image is contained in $R[H \times Y]$, as required.
\end{proof}

\begin{cor}
\label{cor:fpinduction}
If $H$ is a closed subgroup of a t.d.l.c. group $G$ and $M$ is a discrete $\RH$-module, $M$ has type $\FP_n$ over $\RR$ if and only if $\ind^G_HM = \RG \otimes_H M$ has type $\FP_n$ over $\RR$.
\end{cor}

\begin{thm}
\label{thm:LHS}
Suppose $G$ is a t.d.l.c. group and $N$ is a normal closed subgroup. Suppose $N$ has type $\FP_m$ over $\RR$. Then if $Q=G/H$ has type $\FP_n$, $G$ has type $\FP_{\min(m,n)}$, and if $G$ has type $\FP_n$, $Q$ has type $\FP_{\min(m+1,n)}$ over $\RR$. The result also holds if we replace type $\FP_\bullet$ with type $\F_\bullet$.
\end{thm}

We prove a more general result, with the theorem as a special case. Let $G$ be a $k$-group, $N$ closed and normal, and $Q=G/N$. Suppose the quotient map $G \to Q$ has a continuous section (recall this is always the case for t.d.l.c. groups). Suppose $N$ has type $\FP_m$ over $\RR$. Let $A$ be a $k$-$\RR[Q]$-module, which we think of also as a $k$-$\RG$-module by restriction.

\begin{prop}
\begin{enumerate}[(i)]
\item If $A$ has type $\KP_n$ over $\RG$, it has type $\KP_{\min(m+1,n)}$ over $\RR[Q]$.
\item If $A$ has type $\KP_n$ over $\RR[Q]$, it has type $\KP_{\min(m,n)}$ over $\RG$.
\end{enumerate}
\end{prop}
\begin{proof}
First note that $A$ is compactly generated as an $\RG$-module if and only if it is compactly generated as an $\RR[Q]$-module: this follows from the canonical map $\RG \to \RR[Q]$ being $CH$-split. If $A$ is not compactly generated we are done; assume it is. Take a short exact sequence 
\begin{equation}\label{eq:ses}
0 \to K \to P \twoheadrightarrow^{CH} A\to 0,
\end{equation}
 with $P$ a compactly generated projective $\RR[Q]$-module. Now $P$ has type $\KP_m$ as $k$-$\RG$-module: indeed, $\RR[G/N]$ has type $\KP_m$ over $\RR$ by Lemma~\ref{lem:fpinduction}, and the rest follows easily.
\begin{enumerate}[(i)]
\item Use induction on $n$. When $n=0$ we are done. The $\RR[Q]$-module $K$ has type $\KP_{\text{min}(m,n-1)}$ over $\RG$ by Proposition \ref{prop:KPn} so by hypothesis it has type $\KP_{\text{min}(m+1,\text{min}(m,n-1))} = \KP_{\text{min}(m,n-1)}$ over $\RR[Q]$. Therefore as an $\RR[Q]$-module $A$ has type $\KP_{\text{min}(m,n-1)+1} = \KP_{\text{min}(m+1,n)}$ over $\RR[Q]$ by \eqref{eq:ses}.
\item Use induction on $n$. When $n=0$ we are done. $K$ has type $\KP_{n-1}$ over $\RR[Q]$ by Proposition \ref{prop:KPn} so by hypothesis it has type $\KP_{\text{min}(m,n-1)}$ over $\RG$. Since $P$ has type $\KP_m$ over $\RG$, by Proposition~\ref{prop:KPn}, $A$ has type $\KP_{\text{min}(m,n)}$ over $\RG$.
\end{enumerate}
\end{proof}

Since compactly presented t.d.l.c. groups are closed under extensions, and under quotients by compactly generated subgroups, by \cite[Proposition 8.A.10]{cdlh}, the result also holds for type $\F_\bullet$.

\begin{rem}\label{rem:improve}
Other sources in the literature, when $K$ has type $\FP_m$ and $G$ has type $\FP_n$, say only that $G/K$ has type $\FP_{\min(m,n)}$, so we improve this by $1$. (The other sources are speaking of abstract groups, but our proof applies there without change.) Note that this new statement fits much more comfortably alongside analogous results about finite presentability.
\end{rem}

As for type $\FP_n$ and type $\KP_n$, we can compare the type $\F_n$ and type $\K_n$ conditions for a t.d.l.c. group $G$.
\begin{thm}\label{thm:fnkn}
If $G$ has type $\F_n$, it has type $\K_n$.
\end{thm}
\begin{proof}
If $G$ has type $\F_n$, we can copy the argument of \cite[Lemma 4.1]{Luck}, replacing ``finite'' with ``compact'' in the appropriate places and recalling that compact groups have type $\K_\infty$. Given a contractible $(G,\eu C)$-CW-complex $X$ with finite $n$-skeleton, this procedure is essentially taking a cofibrant replacment of $X$ in the compact-open model structure on the category of $G$-KW-complexes (see \cite{Ged}).
\end{proof}

\begin{ques}
Is the converse true?
\end{ques}

\begin{rem}
[Type $\FP_n$ vs Type $\textup{CP}_n$] In \cite{AT} Abels and Tiemeyer introduced compactness properties $\textup{C}_n$ and $\textup{CP}_n$ for arbitrary locally compact groups, which generalise finiteness properties for discrete groups. Even though the starting point of their definition is (the discrete versions of) Brown's criteria we do not know at this stage how their compactness properties relate to our finiteness properties over the class of t.d.l.c. groups. Nevertheless, a relation of their property $\textup{C}_\infty$ to the above $\F_\infty$-property has been promised in \cite{sauert} to be discussed in forthcoming work of the authors.
\end{rem}

\subsection{Almost compactly presented t.d.l.c. groups} 
Let $((\ca A,\Lambda),\phi)$ be a generalised presentation of $G$. Notice that the kernel $N=\ker(\phi)$ is a discrete free subgroup of $ \pi_1(\ca A,\Lambda)$. Indeed $N$ is discrete because it intersects trivially the vertex groups $A_v$ of $(\ca A,\Lambda)$ - which are open in $ \pi_1(\ca A,\Lambda)$ - and, furthermore, $N$ if free since it acts freely on the Bass-Serre tree of $ \pi_1(\ca A,\Lambda)$. Moreover, if $N$ is both normal and discrete, the centraliser of $x$ in $ \pi_1(\ca A,\Lambda)$ is open whenever $x\in N$, and so one has that
\begin{equation}
\bar{R}(\phi)=\frac{N}{[N,N]}\otimes_{\Z}R
\end{equation}
is a discrete left $\RG$-module, which is named the {\it relation module} of the generalised presentation $((\ca A,\Lambda),\phi)$ over $R$.

A t.d.l.c. group $G$ is said to be {\it almost compactly presented} if there exists a generalised presentation $((\ca A,\Lambda),\phi)$ of $G$ such that $\Lambda$ is connected and finite and the corresponding relation module $\bar{R}(\phi)$ is a finitely generated $\R[G]$-module. Clearly,
$$\text{compactly presented}\Rightarrow\text{almost compactly presented}\Rightarrow\text{compactly generated};$$ the converse implications fail because, as is well known, there are discrete groups giving counterexamples.

\begin{prop}
A t.d.l.c. group $G$ is almost compactly presented if and only if $G$ is of type $\FP_2$ over $\Z$.
\end{prop}
\begin{proof}
Since $G$ is of type $\FP_1$ if and only if $G$ is compactly generated, we can assume $G$ to be compactly generated without loss of generality. Let $((\ca A,\Lambda),\phi)$ be a generalised presentation of $G$ such that $\Lambda$ is finite and connected. For $\Pi=\pi_1(\ca A,\Lambda)$, let $\ca T$ denote the Bass-Serre tree of $\Pi$. Put $N=\ker(\phi)$ and notice that the partial simplicial chain complex of the quotient graph $\Gamma=\ca T\quot N$ given by
$$\xymatrix{E(\Gamma) \otimes_{\Z} R \ar[r]^{\partial} & V(\Gamma) \otimes_{\Z} R \ar[r] & R\ar[r] & 0}$$
is a partial projective resolution of $R$ of finite type in $\RGdis$. Since $\bar{R}(\phi)\cong\ker(\partial)$ (cf. \cite[Fact 5.1(a)]{it:ratdiscoh}), the result follows.
\end{proof}
\section{Finiteness criteria for t.d.l.c. groups}\label{s:criteria}
\subsection{Bieri's criterion}
\label{Bieri}

We now investigate the homological behaviour of these finiteness conditions over $\Q$.

Let $\Gamma$ be a directed graph with no loops. Then limits and colimits can be regarded as functors from the category of ($\Gamma$-diagrams in the category $\QGdis$) into $\QGdis$ and denoted by $\dd\lim$ and $\dd\colim$ respectively. One speaks about {\it exact limits} or {\it exact colimits} whenever the graphs $\Gamma$ are such that $\dd\lim$ and $\dd\colim$ are exact functors. E.g., $\dd\prod$ can be regarded as the exact limit (cf. Fact~\ref{fact:lim}) over a graph $\Gamma$ consisting of a set of vertices and no edges. Dually, the direct limit is a special case of exact colimit. For every discrete left $\QG$-module $A$ and all $k\geq0$ one has easily that
\begin{enumerate}[(a)]
\item the functor $\dTor_k^G(\argu,A)$ commutes with exact colimits;
\item the functor $\dExt^k_G(A,\argu)$ commutes with exact limits.
\end{enumerate}
In this section we prove the analogue in the context of t.d.l.c. groups of Bieri's famous criterion for modules of type $\FP_n$. Those modules turn out to be the discrete $\QG$-modules making $\dTor$ commute with exact limits. The proof is essentially the one provided by Bieri with some adjustment such as the following lemma.
\begin{lem}
\label{lem:lim&proj}
Let $P$ be a finitely generated proper discrete permutation $\QG$-module. Then $\argu\otimes_G P$ commutes with limits. Namely, the natural homomorphism $$(\dd\lim M_*)\otimes_G P \to \lim(M_*\otimes_G P)$$ is an isomorphism of $\Q$-modules.
\end{lem}
\begin{proof}
Since $\argu\otimes_G\argu$ and $\lim$ are additive functors, we can assume $P$ to be transitive, i.e., $P=\Q[G/U]$. Since
\begin{equation}\label{eq:eq1}
\argu\otimes_G\ind_U^G(\Q)\cong \rst^G_U(\argu)\otimes_U\Q\colon\QGdis\to\Qmod
\end{equation}
are naturally isomorphic functors, one has to prove that limits are preserved by taking co-invariants $(\argu)_U$.
Since $U$ is profinite, the canonical map
\begin{equation}
\phi_{A,U}\colon A^U\to A_U,\quad \phi_{A,U}(a)=a\otimes_U 1,\quad\forall a\in A
\end{equation}
is an isomorphism for every $A\in\ob(\QGdis)$ (cf. \cite[Prop. 4.2]{it:ratdiscoh}). Now it suffices to notice that limits are preserved by taking invariants.
\end{proof}
\begin{thm}[Bieri's criterion]
\label{thm:Bieri}
Let $A$ be a discrete (left) $\QG$-module. Then the following are equivalent:
\begin{enumerate}[(a)]
\item $A$ is of type $\FP_n$, $n\in\N\cup\{\infty\}$.
\item For every exact limit the natural homomorphism
\begin{equation}
\dTor_k^G(\dd\lim M_*,A)\to \lim\dTor^G_k(M_*,A)
\end{equation}
is an isomorphism for $k\leq n-1$ and an epimorphism for $k=n$.
\item For an arbitrary direct product $\dd\prod \biB(G)$ of copies of the rational discrete standard bimodule the natural map
\begin{equation}
\dTor_k^G(\dd\prod \biB(G),A)\to\prod\dTor^G_k(\biB(G),A)
\end{equation}
is an isomorphism for $k\leq n-1$ and an epimorphism for $k=n$.
\end{enumerate}
\end{thm}
\begin{proof} (a)$\Rightarrow$(b) Choose a projective resolution $P\to A$ such that the discrete modules $P_k$ are finitely generated permutation modules with compact open stabilisers for all $k\leq n$. By the previous Lemma, the natural homomophism
$$(\dd\lim M_*)\otimes_G P\to \lim(M_*\otimes_G P)$$
is an isomorphism for all $k\leq n$. Since $\lim$ is assumed to be exact it commutes with the homology functor and $(b)$ follows by diagram chasing.

(b)$\Rightarrow$(c) is trivial.

(c)$\Rightarrow$(a) Let $n=0$ and recall that there is an isomorphism
\begin{equation}\label{eq:theta1}
\theta_A\colon\biB(G)\otimes_G A\to A,
\end{equation}
natural in $A$ (cf. \cite[\S 4.2]{it:ratdiscoh}) that can be described by
\begin{equation}\label{eq:theta2}
\theta_A(Ug\otimes_G a)=\frac{1}{|U:U\cap W|}\sum_{r\in \ca R}r\cdot a,\quad\forall\ g\in G, \forall a\in A,
\end{equation}
where $W=\stab(g\cdot m)$ and $\ca R$ is a set of representatives of $U\cap W$ in $U$.

Now take $A$ itself as an index set. By assumption and \eqref{eq:theta1} the map
\begin{equation}
\mu\colon (\dd\prod_{a\in A} \biB(G))\otimes_G A\to \prod_{a\in A} A
\end{equation} is an epimorphism, and so the diagonal in $\prod A$ has a preimage
\begin{equation}
c=\sum_{i=1}^k (\prod_{a\in A}\lambda^a_i)\otimes_G a_i,
\end{equation}
with $\lambda^a_i\in\biB(G)$ and $a_i\in A$ for all $i=1,\ldots, k$ and $a\in A$. In particular,
\begin{equation}\label{eq:prod}
\begin{split}
\prod_{a\in A} a=\mu(c)&=\sum_{i=1}^k (\mu((\prod_{a\in A}\lambda^a_i)\otimes_G a_i))=\\
                                   &=\sum_{i=1}^k(\prod_{a\in A}(\theta_A(\lambda^a_i\otimes_G a_i))=\prod_{a\in A}\sum_{i=1}^k(\theta_A( \lambda^a_i\otimes_G a_i)).
\end{split}
\end{equation}
Note that $\theta_A(\lambda^a_i\otimes_G a_i)$ is in the $\QG$-submodule generated by $a_i$. Thus, by comparing component-wise in \eqref{eq:prod}, the description \eqref{eq:theta2} shows that $A$ is finitely generated by $\{a_1,\ldots,a_k\}$, i.e., $A$ is of type $\FP_0$.

Let $n>0$. As before $A$ is finitely generated, so it admits a presentation $$\xymatrix{0\ar[r]&K\ar[r]&P\ar[r]&A\ar[r]&0}$$ where $P$ is a finitely generated proper discrete permutation module. It suffices then to prove that $K$ is of type $FP_{n-1}$. To this end let $\dd\prod\biB(G)$ be an arbitrary product of copies of $\biB(G)$. By naturality one may consider the following commutative diagram
$$\resizebox{\hsize}{!}{
\xymatrix@C=3mm{
\dTor^G_n(\dd\prod\biB(G) ,P)\ar[d]^{\wr}\ar[r]&\dTor_{n}(\dd\prod\biB(G),A)\ar@{->>}[d]\ar[r]&\dTor_{n-1}^G(\dd\prod\biB(G),K)\ar[d]\ar[r]&\dTor_{n-1}^G(\dd\prod\biB(G),P)\ar[r]\ar[d]^{\wr}      &\dTor_{n-1}^G(\dd\prod\biB(G),A)\ar[d]^\wr\\
\prod\dTor^G_n( \biB(G),P)\ar[r]             &\prod\dTor^G_{n}(\biB(G),A) \ar[r]             &\prod\dTor_{n-1}^G( \biB(G),K)\ar[r]              &\prod\dTor_{n-1}^G( \biB(G),P)\ar[r]   &\prod\dTor_{n-1}^G( \biB(G),A)
}}$$
where the first and fourth maps are isomorphisms by Lemma~\ref{lem:lim&proj}. Finally the 5-Lemma yields that $\dTor_k^G(\prod P_*, k)\to\prod\dTor_k(P_*,K)$ is an isomorphims for $k\leq n-1$ and an epimorphism for $k=n$. By the induction hypothesis one conclude that $K$ is of type $\FP_{n-1}$.
\end{proof}
\begin{fact}\label{rem:bieri} Condition $(b)$ in the previous result is equivalent to one/both of the following conditions:

\noindent $(b')\ \mu\colon(\dd\prod \biB(G))\otimes_G A\stackrel{\sim}{\to} \prod A$ and
$\dTor^G_k(\dd\prod \biB(G),A)=0$ \mbox{for $1\leq k\leq n-1$.}

\noindent $(b'')\ A$ is finitely presented and $\dTor^G_k(\dd\prod \biB(G),A)=0$ for $1\leq k\leq n-1$.
\end{fact}
\begin{rem}
Bieri's cohomological criterion - i.e., $\dExt$ commutes with exact colimits - also works as in the abstract case without any adjustment. Indeed, it has been already used in \cite{ilaria}.
\end{rem}


\subsection{Brown's criteria}
A discrete $G$-CW-complex $X$ is said to be {\it $n$-good for $G$ over $\RR$} if the following two conditions hold:
\begin{itemize}
\item[(G1)] $X$ is acyclic in dimensions $<n$, i.e., the reduced homology $\widetilde H_i(X;R) = 0$ for
$i<n$.
\item[(G2)] For $0\leq p\leq n$ and every $p$-cell $\sigma$, the stabiliser $G_{\sigma}$ - which is a t.d.l.c. group - is of type $\FP_{n-p}$ over $R$.
\end{itemize}
Such an $X$ always exists. For example, we could take $X$ to be the universal discrete $G$-CW-complex with compact stabilisers (cf. \cite[\S 6.2]{it:ratdiscoh}) of $G$, in which case (G1) and (G2) both hold for all $n$ for trivial reasons.
\begin{prop}
\label{Brownelementary}
Suppose that a t.d.l.c. group $G$ admits an $n$-good $G$-CW-complex $X$ over $R$ of type $F_n$, i.e., such that the $n$-skeleton has finitely many $G$-orbits. Then $G$ is of type $\FP_n$.
\end{prop}
\begin{proof}
The proof of \cite[Proposition~1.1]{brown:fp} can be transferred verbatim.
\end{proof}
\begin{rem}
Notice that \cite[Proposition~6.6]{it:ratdiscoh} can be deduced from this result considering the topological realisation of a discrete simplicial $G$-complex.
\end{rem}
A filtration $\{X_\alpha\}$ of $X$ is said to be of {\it finite $n$-type} if each $G$-subcomplex $X_\alpha$ is of type $F_n$. Moreover, the filtration is {\it $\mathrm{H}$-essentially trivial in dimension $n$} if for all $\alpha$ there is some $\beta \geq \alpha$ such that the canonical map $\mathrm{H}_n(X_\alpha,R) \to \mathrm{H}_n(X_\beta,R)$ is trivial, or, equivalently, the $n$th homology group satisfies $\varinjlim_{\alpha}\prod_J\mathrm{H}_n(X_\alpha,R)=0$, for any index set $J$. The following theorem can equivalently be stated using reduced homology, cf. \cite[Theorem 2.2]{brown:fp} -- but note that the proof there works by exploiting Bieri's criterion, which we avoid as our analogue of Bieri's criterion only holds over $\Q$.
\begin{thm}[Brown's criterion for $\FP$-properties]\label{thm:Brown1}
Let $G$ be a t.d.l.c. group and $X$ an $n$-good discrete $G$-CW-complex over $R$ with a filtration $\{X_\alpha\}$ of finite $n$-type. Then $G$ has type $\FP_n$ over $R$ if and only if $\{X_\alpha\}$ is $\mathrm{H}$-essentially trivial in dimension $1 \leq k<n$.
\end{thm}
\begin{proof}
The case $n=0$ is trivial; assume $n>0$.

Consider the cellular chain complex of $X_\alpha$: if $G$ has type $\FP_n$, using Corollary \ref{cor:FPn}, an easy induction shows that $Z_{k}(X_\alpha,R)$, the $k$-cycles of this complex, is finitely generated for $k < n$, so $\mathrm{H}_k(X_\alpha,R)$ is too. But $\varinjlim_\alpha \mathrm{H}_k(X_\alpha,R) = 0$ for $k < n$, so there is some $\beta \geq \alpha$ such that the images of the finitely many generators of $\mathrm{H}_k(X_\alpha,R)$ are $0$, so the image of the whole thing is $0$.

Conversely, suppose $\{X_\alpha\}$ is $\mathrm{H}$-essentially trivial in dimension $k<n$. For some fixed $\alpha$, find a $\beta \geq \alpha$ such that the maps $\mathrm{H}_k(X_\alpha,R) \to \mathrm{H}_k(X_\beta,R)$ are trivial for $k<n$. Assume inductively that $G$ has type $\FP_{n-1}$. We may therefore attach finitely many orbits of $k$-cells, $k<n$, with compact open stabilisers to $X_\beta$ to kill the homology $\mathrm{H}_k(X_\beta,R)$ for $k<n-1$, to get a new complex, $X'_\beta$, still with finite $n$-skeleton. Attach $k$-cells, $k \leq n-2$, to $X_\alpha$ to get a new complex $X'_\alpha$ with the same $(n-2)$-skeleton as $X'_\beta$. Notice that the new cells do not affect $\mathrm{H}_{n-1}(X_\alpha,R)$, so the induced map $\mathrm{H}_{n-1}(X'_\alpha,R) \to \mathrm{H}_{n-1}(X'_\beta,R)$ is trivial.

To deduce that $G$ has type $\FP_n$, it remains to show $Z_{n-1}(X'_\alpha,R)$ is finitely generated. Then we can apply Proposition \ref{Brownelementary} to $X'_\alpha$, after attaching finitely many orbits of cells to kill $\mathrm{H}_{n-1}(X'_\alpha,R)$. Now consider the commutative diagram
\[\xymatrix{
B_n(X'_\alpha,R) \ar@{>->}[r] \ar[d] & Z_{n-1}(X'_\alpha,R) \ar@{->>}[r] \ar[d] & \mathrm{H}_{n-1}(X'_\alpha,R) \ar@{=}[d] \\
B_n(X'_\beta,R) \ar@{>->}[r] \ar@{=}[d] & P \ar@{->>}[r] \ar[d] & \mathrm{H}_{n-1}(X'_\alpha,R) \ar[d]^{0} \\
B_n(X'_\beta,R) \ar@{>->}[r] & Z_{n-1}(X'_\beta,R) \ar@{->>}[r] & \mathrm{H}_{n-1}(X'_\beta,R),}\]
where $B_n()$ denotes the $n$-boundaries and $P$ is defined by the top-left square being a push-out. Observe that the rows are short exact, so the map of homology being trivial implies there is a map $P \to B_n(X'_\beta,R)$ making the diagram commute, and hence $P$ is a retract of the finitely generated module $B_n(X'_\beta,R)$, so itself finitely generated. Now we have an exact sequence $$0 \to B_n(X'_\alpha,R) \to B_n(X'_\beta,R) \oplus Z_{n-1}(X'_\alpha,R) \to P \to 0$$ with $P$ and $B_n(X'_\alpha,R)$ finitely generated; therefore $B_n(X'_\beta,R) \oplus Z_{n-1}(X'_\alpha,R)$ and hence $Z_{n-1}(X'_\alpha,R)$ are too.
\end{proof}
We now prove a similar condition for compact presentability. We will need \cite[Theorem 1']{brown:presentations}. We quote the result in the notation of \cite{brown:presentations}; see there for the definitions.
\begin{thm}\label{thm:bb}
Let $G$ be an abstract group and $X$ a simply-connected $G$-CW-complex. Then the canonical map $\phi\colon\widetilde{G}\to G$ is surjective and its kernel is the normal subgroup of $\widetilde{G}$ generated by the $r_\tau$ ($\tau\in F$).
\end{thm}
In particular, $\widetilde G$ is the fundamental group of a graph of groups $(\ca G, \Lambda)$ (cf. \cite[\S~3]{brown:presentations}) whose vertex/edge groups correspond
to vertex/edge isotropy groups of $X$ in $G$. Moreover, for every vertex group $G_v$ in $(\ca G, \Lambda)$, the restriction map $\phi\restriction_{G_v}\colon G_v\to G$ is the inclusion map.
\begin{thm}\label{thm:Brown2}
\label{thm:compres}
Let $G$ be a t.d.l.c. group which admits a simply connected discrete $G$-CW-complex satisfying:
\begin{enumerate}[(a)]
\item Every vertex isotropy group is compactly presented.
\item Every edge isotropy group is compactly generated.
\item $X$ has a finite $2$-skeleton mod $G$.
\end{enumerate}
Then $G$ is compactly presented.
\end{thm}
\begin{proof}
Let $\phi\colon\widetilde G\to G$ be the map of (abstract) groups given by Theorem~\ref{thm:bb}. Since all the isotropy groups are open in $G$, we can endow $\widetilde G$ with the unique group topology such that the inclusions $G_v\to\widetilde G$ and $G_e\to\widetilde G$ are topological isomorphisms onto open subgroups of $\widetilde G$ for all vertices $v$ and edges $e$ in the graph $\Lambda$ defining $\widetilde G$. Therefore, $\widetilde G$ is compactly presented by $(a)$ and $(b)$; see \cite[Propositions 8.B.9 and 8.B.10]{cdlh} for all the details. In particular, $\phi$ is a continuous surjective map of t.d.l.c. groups by construction. To conclude the proof it suffices to notice that we now have a short exact sequence
\begin{equation}
0\to N\to\widetilde G\to G\to 0,
\end{equation}
of t.d.l.c. groups and continuous homomorphisms, where the kernel $N$ is finitely generated as normal subgroup of $\widetilde G$, by $(c)$. Indeed, \cite[Prop.~8.A.10]{cdlh} shows that $G$ is compactly presented.
\end{proof}
As a consequence, we now obtain an analogue of \cite[Theorem 3.2]{brown:fp}.
\begin{thm}[Brown's criterion for compact presentability]\label{thm:Brown3}
Let $X$ be a simply connected discrete $G$-CW-complex such that the vertex stabilisers are compactly presented and the edge stabilisers are compactly generated. Let $\{X_\alpha\}$ be a filtration of $X$ such that each $X_\alpha$ has a finite $2$-skeleton mod $G$, and let $v \in \bigcap X_\alpha$ be a basepoint. If $G$ is compactly generated, then $G$ is compactly presented if and only if the direct system $\{\pi_1(X_\alpha,v)\}$ is essentially trivial.
\end{thm}
\begin{proof}
As in the proof of \cite[Theorem 3.2]{brown:fp}, we may reduce to the case where every $X_\alpha$ is connected. Write $H_\alpha$ for the image of $G$ in $\Aut(X_\alpha)$, with the quotient topology from $G$. Let $K_\alpha$ denote the group of homeomorphisms of the universal cover of $X_\alpha$ which are in the preimage of some element of the action of $H_\alpha$ on $X_\alpha$. Then one has a canonical short exact sequence
\begin{equation}\label{eq:ses2}
0 \to \pi_1(X_\alpha) \to K_\alpha \to H_\alpha \to 0
\end{equation}
of abstract groups, where $\pi_1(X_\alpha)$ maps to the deck transformations; since $H_\alpha$ acts on $X_\alpha$ (and hence $\pi_1(X_\alpha)$) with open stabilisers, it follows from \cite[Proposition 8.B.4]{cdlh} that there is a unique topology making $K_\alpha$ a t.d.l.c. group such that \ref{eq:ses2} is a short exact sequence of t.d.l.c. groups, with $\pi_1(X_\alpha)$ discrete.

Now we may pull back by the quotient $G \to H_\alpha$ to get a short exact sequence
\begin{equation}
0 \to \pi_1(X_\alpha) \to G_\alpha \to G \to 0
\end{equation}
of t.d.l.c. groups, and our hypotheses imply that $G_\alpha$ acts on the universal cover of $X_\alpha$ with compactly presented vertex stabilisers and compactly generated edge stabilisers. Therefore $G_\alpha$ is compactly presented by Theorem \ref{thm:compres}.

Finally, note that if $G$ is compactly presented, each $\pi_1(X_\alpha)$ is compactly normally generated in $G_\alpha$ (cf. \cite[Prop 8.A.10]{cdlh}), and discrete, hence finitely normally generated. Now the rest of the proof follows by the same argument as \cite[Theorem 3.2]{brown:fp}.
\end{proof}
As in \cite[Corollary 3.3]{brown:fp}, we immediately get:
\begin{cor}
\label{cor:highlyconnected}
Let $X$ be a contractible discrete $G$-CW-complex such that the stabiliser of every cell is of type $\F_\infty$. Let $\{X_j \}_{j \in \mathbb{N}}$ be a filtration of $X$ such that each $X_j$ is of finite type. Suppose that the connectivity of the pair $(X_{j+ 1}, X_j)$ tends to $\infty$ as $j$ tends to $\infty$. Then $G$ is of type $\F_\infty$.
\end{cor}
\begin{ex} In the discrete case, this criterion has been used for example to prove that  Highman--Thompson groups \cite{brown:fp} and Brin--Thompson groups \cite{fluch} are finitely presented and of type $\FP_\infty$.  The Higman--Thompson families of groups $F_{qr}$ and $V_{qr}$ embed into certain tree almost automorphisms groups, denoted sometimes by $A^D_{qr}$, which are totally disconnected and locally compact. In \cite{sauert}, the t.d.l.c. groups $A^D_{qr}$ has been recently proved to be of type $\F_\infty$, and Corollary \ref{cor:highlyconnected} immediately gives an alternative proof of this fact. Indeed, \cite[4.5]{sauert} shows that the contractible proper smooth $A^D_{qr}$-CW-complex $\mathcal{Q}$, together with the filtration $\mathcal{Q}(k)$, satisfies the hypotheses of the corollary, and we get:
\end{ex}
\begin{cor}
The groups $A^D_{qr}$ have type $\F_\infty$.
\end{cor}
\begin{rem}
All the results stated in this section continue to hold when $X$ is a discrete $G$-complex, i.e., $G$-actions with inversions are allowed.
\end{rem}

\section{Quasi-isometric invariance}\label{s:qii}
\subsection{Large-scale language} Let $X$ be a set.  A {\it pseudo-metric} on  $X$ is a map
$d\colon X\times X\to[0,+\infty[$
such that $d(x,x)=0$, $d(x,x')=d(x',x)$, and $d(x,x'')\leq d(x,x')+d(x',x'')$ for all $x,x',x''\in X$. A {\it pseudo-metric space} is a pair $(X,d)$ consisting of a set $X$ and a pseudo-metric $d$ on $X$.
A map $f\colon (X,d)\to (X',d')$ of pseudo-metric spaces is said to be {\it large-scale Lipschitz} if there exist $\mu>0$ and $\alpha\geq 0$ satisfying
$d'(f(x_1),f(x_2))\leq \mu\cdot d(x_1,x_2)+\alpha, \forall x_1,x_2\in X.$
Maps $f,g\colon X\to X'$ are said to be {\it close} (or {\em at bounded distance}) if $$\sup_{x\in X}d'(f(x),g(x))\leq\infty,$$
and they are denoted by $f\sim f'$. A {\it quasi-isometry} between pseudo-metric spaces $X$ and $X'$ is a large-scale Lipschitz map $f \colon X \to X'$ together with a large-scale Lipschitz map $f'\colon X' \to X$ such that $$f\circ f'\sim \iid_{X'}\quad\text{and}\quad f'\circ f\sim \iid_X;$$ the map $f'$ is called an {\it inverse quasi-isometry to $f$}.
Two pseudo-metric spaces are {\it quasi-isometric} if there exists a quasi-isometry from one to the other (and thus conversely).

More generally, $X$ is a {\it quasi-retract} of $X'$, denoted by $X\preceq X'$, if there exist large-scale Lipschitz maps $X\stackrel{i}{\to} X' \stackrel{r}{\to} X$ such that $r\circ i\sim\iid_X$. The pair $(i,r)$ is called a {\it quasi-retraction} of $X'$ to $X$.
\begin{fact} The following properties hold:
\begin{enumerate}
\item if $X$ is quasi-isometric to $X'$ than $X\preceq X'$ and $X'\preceq X$;
\item if $X\preceq X'$ and $X'\preceq X''$, then $X\preceq X''$;
\item quasi-retraction is preserved by quasi-isometries.
\end{enumerate}
\end{fact}
\begin{ex}[T.d.l.c. groups as pseudo-metric spaces]
By  \cite[Prop. 1.D.2]{cdlh}, every compactly generated locally compact group can be regarded as a pseudo-metric space, well-defined up to quasi-isometry. 

Whenever $G$ is totally disconnected we can consider $X$ to be a Cayley-Abels graph of $G$ and $d_X$ the combinatorial metric on $X$. The natural action of $G$ on $X$ is geometric and, therefore, it induces a geodesically adapted pseudo-metric $d_G$ on $G$ such that the orbit map $(G,d_G)\to (X,d_X)$ is a quasi-isometry (cf. \cite[Thm. 4.C.5]{cdlh}). In particular, we deduce that two compactly generated t.d.l.c. groups $ G$ and $H$ are {\it quasi-isometric} if some/any Cayley-Abels graph for $G$ and $H$, respectively, are quasi-isometric. Analogously for quasi-retraction.
\end{ex}

\subsection{The Rips' complex of a compactly generated t.d.l.c. group}
Let $G$ be a compactly generated t.d.l.c. group and $\Gamma$ a Cayley-Abels graph of $G$. For every positive integer $d$, the {\it Rips' complex $P_d(\Gamma)$ of $\Gamma$} is the simplicial complex whose $n$-simplices correspond to $(n + 1)$-tuple $\{v_0,v_1,...,v_n\}$ of vertices of $\Gamma$ satisfying
$$d(v_i,v_j) \leq d\quad\text{for}\ 0 \leq i < j \leq n.$$ Since $\Gamma$ is a connected graph of bounded valency such that the $G$-action is continuous, proper and vertex-transitive, one has the following properties.
\begin{fact}\label{fact:rips} Let $G$ be a compactly generated t.d.l.c. group and $\Gamma$ a Cayley-Abels graph associated to $G$. Then
\begin{enumerate}[(a)]
\item $P_d(\Gamma)$ is locally finite and finite dimensional for all $d$;
\item $G$ acts simplicially and cocompactly on $P_d(\Gamma)$;
\item $G$ acts transitively on the vertices of $P_d(\Gamma)$;
\item the stabiliser of any simplex in $P_d(\Gamma)$ is compact and open in $G$;
\item the simplicial complex $P_\infty(\Gamma)=\bigcup_d P_d(\Gamma)$ is contractible.
\end{enumerate}
\end{fact}
\begin{cor}
Let $\Gamma$ be a Cayley-Abels graph of a compactly generated t.d.l.c. group $G$. The discrete $G$-CW-complex $|P_\infty(\Gamma)|$ is an $n$-good $G$-complex for all $n\geq 1$. Moreover, $\{|P_d(\Gamma)|\}_d$ is a filtration of $|P_\infty(\Gamma)|$ of finite $n$-type for all $n\geq1$.
\end{cor}
\begin{proof}
It is a direct consequence of Fact~\ref{fact:rips} (recall that every totally disconnected compact group is of type $\FP_\infty$).
\end{proof}
\begin{thm}\label{thm:quasiretract}
Let $G, G'$ be compactly generated t.d.l.c. groups. Suppose that $G$ is a quasi-retract of $G'$ and let $n\geq 2$. Then
\begin{enumerate}
\item if $G'$ is of type $\FP_n$ over $\RR$, then so is $G$;
\item if $G'$ is of type $\F_n$, then so is $G$.
\end{enumerate}
\end{thm}
\begin{proof} The proof of \cite[Theorem 8]{alonso} can be transferred verbatim. \end{proof}
\begin{cor} Let $G$ be a compactly generated t.d.l.c. group and $H$ a closed subgroup. If $H$ is a group retract, then
\begin{enumerate}
\item $G$ is of type $\FP_n$ over $\RR$, then so is $H$;
\item $G$ is of type $\F_n$, then $H$ is of type $\F_n$.
\end{enumerate}
\end{cor}
\begin{cor}\label{cor:qi}
Let $G$ and $G'$ be quasi-isometric compactly generated t.d.l.c. groups. Then
\begin{enumerate}
\item $G$ is of type $\FP_n$ over $\RR$ if and only if $G'$ is of type $\FP_n$ over $\RR$.
\item $G$ is of type $\F_n$, then $H$ is of type $\F_n$.
\end{enumerate}
\end{cor}

\subsection{Uniform lattices}
A {\it uniform lattice} $\Gamma$ in $G$ is a cocompact discrete subgroup of $G$. For example, given a finitely generated (discrete) group $H$ and a Cayley graph $X$ for $H$, then $H$ is a uniform lattice in $\textup{Aut}(X)$, which is a (not necessarily discrete) t.d.l.c. group  (see \cite[Proposition 2.3]{moller}).

\medskip

Let $G$ be a t.d.l.c. group and $\Gamma$ a uniform lattice in $G$. By \cite[Prop.~5.C.3]{cdlh}, $G$ is compactly generated if and only if $\Gamma$ is finitely generated. Moreover, the natural inclusion $(\Gamma,\textup d_\Gamma)\subseteq (G,\textup d_G)$ is a quasi-isometry whenever $G$ is compactly generated. Therefore one deduces the following generalisation as a consequence of Corollary~\ref{cor:qi}.
\begin{prop}\label{prop:ulattice}
Let $\Gamma$ be a uniform lattice of $G$. Then $G$ is of type $\FP_n$ (resp. $\F_n$) if and only if the discrete group $\Gamma$ is of type $\FP_n$ (resp. $\F_n$).
\end{prop}

This provides a new technique to investigate finiteness properties of both t.d.l.c. groups and discrete groups. On the other hand, it is worth remarking that there exist t.d.l.c. groups that not contain any uniform lattice, indeed any lattice at all  (e.g., non-unimodular t.d.l.c. groups and Neretin's group).

\begin{ex}
\begin{enumerate}[(i)]
\item Every compactly generated abelian t.d.l.c. group $G$ is of type $\FP$. Indeed, $G=\Z^m\times K$ with $K$ compact, i.e., $\Z^m$ is a uniform lattice in $G$. For $K$ compact and normal, $\ccd_\Q(G)=\ccd_\Q(\Z^m)=m$.
\item If the automorphism group $G$ of a locally finite tree contains a uniform lattice, then $G$ is of type $\FP_\infty$; see \cite[5.C.11]{cdlh}. More generally, every compactly generated unimodular t.d.l.c. group whose rough Cayley graph is quasi-isometric to a tree is of type $\F_\infty$; see \cite[Theorem 3.28]{km}. In particular, we know these groups to be rational duality t.d.l.c. groups; compare with \cite{it:ratdiscoh}.
\item Let $\Gamma$ be a discrete group of type $\FP_n$ or $\F_n$, with $n>0$. Then the automorphism group of a Cayley graph of $\Gamma$ is a (possibly discrete) t.d.l.c. group of the same type.
\end{enumerate}
\end{ex}
\begin{rem}
Proposition~\ref{prop:ulattice} gives a promising new approach to answering Question~2 in \cite{it:ratdiscoh}, by using suitable examples of discrete groups of type $\FP_n$ but not of type $\F_n$.
\end{rem}

\section{T.d.l.c. graph-wreath products}
\label{s:gwp}

\subsection{Polyhedral products and graph-wreath products}

A wreath product is defined for t.d.l.c. groups in \cite{Cornulier}: for $B$ and $H$ two t.d.l.c. groups, $A$ a compact open subgroup of $B$, and $X$ a discrete space on which $H$ acts continuously, the wreath product $B \wr^A_X H$ is defined to be the semidirect product, by $H$, of the semi-restricted power
\begin{equation}
B^{X,A} = \{f \in B^X: f(x) \in A \text{ for all but finitely many } x\}.
\end{equation}
There is a unique t.d.l.c. group topology on $B^{X,A}$ making the obvious embedding of $A^X$ a topological isomorphism onto a compact open subgroup: this defines a t.d.l.c. topology on the wreath product.

Here we define graph-wreath products, analogously to the graph-wreath products of abstract groups defined in \cite{KropMart}, of which the wreath product is a special case; this generality will make the proof of Theorem \ref{F2necessary} much easier. We fix here notation for the building blocks of a graph-wreath product that will be used throughout this section. Suppose $B$ is a t.d.l.c. group with compact open subgroup $A$, and $\Gamma$ is a graph with vertices $V$. We define the polyhedral product, $B^{\Gamma,A}$, as follows. As an abstract group, $B^{\Gamma,A}$ is the free product of the product $A^V$, and a copy $B_v$ of $B$, for each $v \in V$, subject to:
\begin{enumerate}[(i)]
\item if $(a_v) \in A^V$ with $a_v=1$, for all $v \neq w$, then $(a_v)$ is identified with $a_w \in B_w$;
\item if $(a_v) \in A^V$ with $a_w=1$, some $w$, and $g \in B_w$, then $g$ commutes with $(a_v)$;
\item if $v_1,v_2 \in V$ are joined by an edge in $\Gamma$, $g \in B_{v_1}$, and $h \in B_{v_2}$, then $gh = hg$.
\end{enumerate}

Now, if $H$ is a t.d.l.c. group with a continuous discrete action on $\Gamma$, we define the graph-wreath product $B \gwr_\Gamma^A H$ to be the semi-direct product of $B^{\Gamma,A}$ by $H$, via the action of $H$ on $\Gamma$ permuting the copies of $B$.

\begin{prop}
Think of the copy of $A^V$ in $B^{\Gamma,A}$ as a profinite group, with the product topology from the topologies on the copies of $A$.
\begin{enumerate}[(i)]
\item There is a unique structure of a t.d.l.c. group on $B^{\Gamma,A}$ with $A^V$ a compact open subgroup.
\item There is a unique structure of a t.d.l.c. group on $B \gwr_\Gamma^A H$ that makes it a topological semi-direct product.
\end{enumerate}
\end{prop}
\begin{proof}
(i) follows from \cite[Proposition 8.B.4]{cdlh}. (ii) may be proved in exactly the same way as \cite[Proposition 8.B.4]{Cornulier}.
\end{proof}

Note that the polyhedral product $B^{\Gamma,A}$, where $\Gamma$ is the complete graph on its vertices, is the same as the semi-restricted power $B^{V,A}$ on the vertices of $\Gamma$. So wreath products of t.d.l.c. groups are a special case of graph-wreath products.

Note too that, if $\Gamma_1, \Gamma_2$ are two graphs with the same set of vertices and the edge set of $\Gamma_1$ is contained in the edge set of $\Gamma_2$, the induced map $B \gwr_{\Gamma_1}^A H \to B \gwr_{\Gamma_2}^A H$ is a quotient. Moreover this map is an isomorphism only if the two graphs are equal (see \cite[Lemma 2.4]{Cornulier2}). The same is true if instead $\Gamma_2$ is a quotient graph of $\Gamma_1$.

It would be more comforting if the compact open subgroup $A^V$ of $B^{\Gamma,A}$ looked more like a colimit, as is the case for abstract groups. Indeed, when $\Gamma$ is a discrete graph, we might hope for $B^{\Gamma,A}$ to be something like a coproduct of t.d.l.c. groups. The following proposition tells us not to interpret this hope too literally.

\begin{prop}
The category of t.d.l.c. groups does not have all small coproducts.
\end{prop}
\begin{proof}
If the category of t.d.l.c. groups had coproducts, so would the category of abelian t.d.l.c. groups, because abelianisation is a reflection functor from the former to the latter. But the Pontryagin dual of \cite[Lemma 9]{HM} says that a coproduct $\coprod_i G_i$ exists in the category of abelian t.d.l.c. groups only if, for all but finitely many $i$, $G_i$ has a minimal compact open subgroup -- that is, only if all but finitely many of the $G_i$ are discrete.
\end{proof}

On the other hand, the following construction of a $B^{\Gamma,A}$-space, when $\Gamma$ is discrete, should be thought of as corresponding to building an Eilenberg-MacLane space for a free product of groups as the wedge sum of Eilenberg-MacLane spaces for each of the groups.

See \cite{KropMart} for the definition of a flag complex $L$ of a graph $\Gamma$.

The presentation of $B^{\Gamma,A}$ described above suggests a generalised presentation for $B^{\Gamma,A}$, following the approach of \cite[Proposition 5.10]{it:ratdiscoh}: fix a discrete subset $S$ of $B$ such that $A$ and $S$ together generate $B$; write $A_v$ and $S_v$ for the copies of $A$ and $S$ in each $B_v$. Without loss of generality, we may assume $S = S^{-1}$. Let $\Lambda$ be the graph with one vertex $x$ and a loop for every $s \in \bigcup_v S_v$, which we abusively call $s$. Let $(\Pi,\Lambda)$ be the graph of profinite groups based on $\Lambda$, defined as follows:
\begin{enumerate}[(i)]
\item $\Pi_x = A^V$;
\item $\Pi_s = A^V \cap s^{-1}A^Vs$ for all $s \in \bigcup_v S_v$;
\item $\alpha_s: \Pi_s \hookrightarrow s^{-1}A^Vs \xrightarrow{^s-} A^V$;
\item $\omega_s: \Pi_s \hookrightarrow A^V$.
\end{enumerate}
The proof of \cite[Proposition 5.10(a)]{it:ratdiscoh} shows that $(\Pi,\Lambda)$, together with the obvious map $\pi_1(\Pi,\Lambda) \to B^{\Gamma,A}$, is a generalised presentation of $B^{\Gamma,A}$; the kernel $K$ of this map is then a discrete free group.

Similarly, for each $S_v$, we can imitate \cite[Proposition 5.10]{it:ratdiscoh} to define a generalised presentation $(\Pi_v,\Lambda_v)$ for $B_v$. Write $K_v$ for the kernel of the map $\pi_1(\Pi_v,\Lambda_v) \to B_v$ and $T_v$ for a set of generators of $K_v$. The canonical maps $B_v \to B^{\Gamma,A}$, $S_v \to \bigcup_v S_v$ and $A_v \to A^V$ induce a map $K_v \to K$, and we will identify $T_v$ with its image under this map. Now it is not hard to check that $\bigcup_v T_v$, together with the elements $[s,t]$ for all $s \in S_v$, $t \in S_w$ such that $v,w$ are joined by an edge in $\Gamma$, generate $K$.

Finally, using the construction from the proof of Proposition \ref{prop:typeF2}, we can use this generalised presentation to get a simply connected $B^{\Gamma,A}$-CW-complex $X^\Gamma_2$. Moreover, the $H$-action on $\Gamma$ induces an $H$-action on our choice of generalised presentation -- that is, $H$ acts on the graph of profinite groups $(\Pi,\Lambda)$ and on our choice of generating set of $K$ -- and thus we get a cellular $H$-action on $X^\Gamma_2$ which is compatible with the $B^{\Gamma,A}$-action, and makes $X^\Gamma_2$ into a simply connected $B \gwr_\Gamma^A H$-CW-complex.

Thanks to the Hurewicz theorem, we may add higher cells to kill higher homotopy and get a contractible $B \gwr_\Gamma^A H$-CW-complex. 

\subsection{Sufficient finiteness conditions}

With these definitions in place, we can prove various necessary and sufficient conditions for wreath products and graph-wreath products of t.d.l.c. groups. We will give necessary conditions in the next subsection, and sufficient conditions in this; in both cases, we prove conditions about compact presentation for graph-wreath products, and restrict to wreath products for type $\F_n$ conditions for higher $n$, thanks to the technical issue of constructing the right space for the graph-wreath product to act on.

For the rest of Section \ref{s:gwp}, type $\FP_n$ will always mean over $\Z$ -- this allows us to use Proposition \ref{prop:FvsFP}.

We start by noting that the following analogues of \cite[Proposition 2.1, Theorem 2.2]{Cornulier2} hold here:

\begin{prop}
\label{prop:compres}
\begin{enumerate}[(i)]
\item Suppose that $B$ and $H$ are compactly generated, and that $H$ acts on the vertices of $X$ with finitely many orbits. Then $G = B \gwr^A_\Gamma H$ is compactly generated.
\item Suppose that $B$ and $H$ are compactly presented, and that $H$ acts on the vertices of $\Gamma$ with compactly generated stabilisers and acts on the vertices and edges of $\Gamma$ with finitely many orbits. Then $G = B \gwr^A_\Gamma H$ is compactly presented.
\end{enumerate}
\end{prop}
\begin{proof}
\begin{enumerate}
\item $G$ is generated by $A^V$, a compact set of generators for $H$, and a compact set of generators for $B_v$, for one choice of $v$ from each $H$-orbit of $V$.
\item A presentation can be written out explicitly, as in \cite[Theorem 2.2]{Cornulier2}. The details are left to the reader.
\end{enumerate}
\end{proof}

Now we can produce an analogue, for t.d.l.c. groups, of higher finiteness conditions for wreath products of abstract groups. Here we imitate the proof of \cite[Theorem A]{KropMart}.

\begin{thm}
\label{PeterA}
Suppose that $B$ and $H$ have type $\F_n$, and that for $1 \leq p \leq n$ $H$ acts diagonally on $X^p$ with finitely many orbits and with stabilisers of type $\FP_{n-p}$. Then $G = B \wr^A_X H$ has type $\F_n$.
\end{thm}
\begin{proof}
Let $Y$ be a contractible $(B,\eu C)$-CW-complex of type $\F_n$. The construction in the proof of Proposition \ref{prop:typeF2} shows that we may choose $Y$ to have a single orbit of $0$-cells, one of which is stabilised by $A$ (which we take to be the basepoint $\ast$ of $Y$). We define the finitary product $\bigoplus_X Y$ to be a CW-complex with cells given by the set $$\{(C_x)_{x \in X}: C_x = \ast \text{ for all but finitely many } x\}:$$ $n$-cells are elements $(C_x)$ for which the dimensions of the $C_x$ sum to $n$, and the attaching maps are defined component-wise. This has basepoint $(\ast)_x$.

Since $Y$ is a CW-complex, it is well-pointed: that is, the inclusion $\ast \to Y$ is a (Hurewicz) cofibration. By \cite[Proposition 4.4.14]{AGP}, it follows that $Y$ is contractible by a based homotopy. Now, applying this contracting homotopy componentwise to $\bigoplus_X Y$, we see that $\bigoplus_X Y$ is contractible.

Now $B^{X,A}$ acts componentwise on $\bigoplus_X Y$ and $H$ acts by permuting components, making $\bigoplus_X Y$ into a $G$-CW-complex. For each $p \leq n$, since $Y$ has finitely many $B$-orbits of $q$-cells for all $q \leq p$, and $X^q$ has finitely many $H$-orbits for all $q \leq p$, $\bigoplus_X Y$ has finitely many $G$-orbits of $p$-cells. By Proposition \ref{Brownelementary}, to show $G$ has type $\FP_n$ it suffices to show the stabilisers of this $G$-action in dimension $p$ have type $\FP_{n-p}$.

To see this, we consider a $p$-cell $(C_x)$, where only $C_{x_1}, \ldots, C_{x_j}$ have dimension $>0$. Note that $\stab_B(C_x)$ for $C_x$ any $0$-cell $\neq \ast$ is conjugate to $A$: therefore $\stab_G((C_x)) \cong \stab_G((C'_x))$, where $(C'_x)$ is the $p$-cell defined by $C'_x = C_x$ for $x \in \{x_1, \ldots, x_j\}$ and $C'_x = \ast$ otherwise. We will show $\stab_G((C'_x))$ has type $\FP_{n-p}$.

Let $K$ be the stabiliser in $H$ of $(x_1, \ldots, x_j) \in X^j$: by hypothesis, this has type $\FP_{n-j}$, and $n-j \geq n-p$. Now the stabiliser of $(C'_x)$ in $B^{X,A}$ is $$\prod_x \stab_B(C'_x) = \prod_{i=1}^j \stab_B(C_{x_i}) \times \prod_{x \notin \{x_1, \ldots, x_j\}} A,$$ which is a compact open subgroup of $B^{X,A}$ (and hence has type $\FP_\infty$). It follows that the semidirect product $$(\prod_{i=1}^j \stab_B(C_{x_i}) \times \prod_{x \notin \{x_1, \ldots, x_j\}} A) \rtimes K$$ given by the restriction of $B^{X,A} \rtimes H$ is isomorphic (as topological groups) to a subgroup of finite index in $\stab_G((C'_x))$.

As an extension of a type $\FP_\infty$ group by a type $\FP_{n-p}$ group, this semidirect product has type $\FP_{n-p}$ by Theorem \ref{thm:LHS}, so $\stab_G((C'_x))$ has type $\FP_{n-p}$ too by Lemma \ref{lem:cominvariant}.

Therefore $G$ has type $\FP_n$. For $n=1$ we are done. For $n \geq 2$, $G$ is compactly presented by Proposition \ref{prop:compres}, and so $G$ has type $\F_n$ by Proposition \ref{prop:FvsFP}.
\end{proof}

\begin{rem}
As noted in \cite{KropMart}, we could also prove an analogous condition replacing every instance of type $\F_n$ with type $\FP_n$, similarly to \cite{BCK}, at the cost of notational complexity.
\end{rem}

\subsection{Necessary finiteness conditions}

In the rest of this section we will prove a partial converse to Theorem \ref{PeterA}, starting with the following converse to Proposition \ref{prop:compres}. We assume from now on, in our wreath products $B \wr^A_X H$, that $X \neq \emptyset$ and $B \neq 1$, to avoid trivial counter-examples. Similarly, for graph-wreath products $B \gwr^A_\Gamma H$, we assume $\Gamma \neq \emptyset$ and $B \neq 1$.

\begin{thm}
\label{F2necessary}
\begin{enumerate}[(i)]
\item Suppose the graph-wreath product $G = B \gwr^A_\Gamma H$ of t.d.l.c. groups is compactly generated. Then $B$ and $H$ are compactly generated, and $H$ acts on the vertices of $\Gamma$ with finitely many orbits.
\item Suppose $G = B \gwr^A_\Gamma H$ is compactly presented. Then $B$ and $H$ are compactly presented; $H$ acts on the vertices of $\Gamma$ with compactly generated stabilisers, and acts on the vertices and edges of $\Gamma$ with finitely many orbits.
\end{enumerate}
\end{thm}
\begin{proof}
\begin{enumerate}[(i)]
\item $H$ is a quotient of $G$, hence compactly generated. $B$ embeds in $G$, so is $\sigma$-compact. Similarly, $B^{\Gamma,A}$ is $\sigma$-compact, so $V$ is countable.

Now if $B$ is not compactly generated, it can be written as the union of a nested sequence $(B_n)$ of compactly generated open subgroups, by \cite[Proposition 2.C.3]{cdlh}. We may assume every $B_n$ contains $A$. But then $G$ is the union of the nested sequence of open subgroups $(B_n \gwr^A_X H)$, and hence $G$ is not compactly generated. 

Finally, if $V$ has infinitely many $H$-orbits $V_1,V_2,\ldots$, $G$ is the union of a sequence $(B \gwr^A_{\Gamma_n} H)$, where each $\Gamma_n$ is the subgraph of $\Gamma$ on $\bigcup_{i=1}^nV_n$. For each $n$, $B^{\Gamma_n,A}$ is a closed subgroup $B^{\Gamma_{n+1},A}$, because the inclusion map is split (see \cite[Lemma 2.3]{Cornulier2}). Therefore $B \gwr^A_{\Gamma_n} H$ is a closed subgroup of $B \gwr^A_{\Gamma_{n+1}} H$, and hence $G$ is not compactly generated.

\item By (i), $B$ and $H$ are compactly generated, and $H$ acts on $V$ with finitely many orbits. Now $H$ is compactly presented because it can be written as a quotient of the compactly presented group $B \gwr^A_\Gamma H$ by adding relations which put a compact generating set of $B_x$ equal to $1$, for one $x$ from each $H$-orbit of $V$.

Using Proposition \ref{compprescolimit}, we may write $B$ as the abstract and topological colimit of a sequence $(B_n)$ of compactly presented groups with quotient maps between them. Assume $B$ is not compactly presented, so the sequence does not stabilise. Then $B \gwr^A_\Gamma H$ is the abstract and topological colimit of the sequence $(B_n \gwr^A_\Gamma H)$ of compactly generated groups by Proposition \ref{prop:compres}(i), which does not stabilise, giving a contradiction.

Now \cite[Lemma 2.9]{Cornulier2} holds in the context of t.d.l.c. groups, so we may imitate the proof of \cite[Proposition 2.10]{Cornulier2}. As there, if $\Gamma$ has infinitely many $H$-orbits of edges, or the stabiliser of some vertex is not compactly generated, we may construct a sequence of $H$-graphs $(\Gamma_n)$ with colimit $\Gamma$, so that $B \gwr^A_\Gamma H$ is the abstract and topological colimit of the sequence $(B \gwr^A_{\Gamma_n} H)$ which does not stabilise. Moreover the $\Gamma_n$ all have finitely many $H$-orbits of vertices, so every $B \gwr^A_{\Gamma_n} H$ is compactly generated by Proposition \ref{prop:compres}(i). So Proposition \ref{compprescolimit} gives a contradiction.
\end{enumerate}
\end{proof}

We now restrict once again to wreath products. We follow the general approach of \cite{KropMart}.

\begin{thm}
\label{PeterB}
Suppose $G = B \wr^A_X H$ has type $\F_n$. Suppose for $1 \leq p \leq n-1$ that $H$ acts diagonally on $X^p$ with stabilisers of type $\FP_{n-p}$. Then $B$ and $H$ have type $\F_n$ and $H$ acts diagonally on $X^n$ with finitely many orbits.
\end{thm}
\begin{proof}
$H$ has type $\F_n$ because it is a retract of $G$. For $n \leq 2$ we are done by Theorem \ref{F2necessary}; assume $n \geq 3$. By induction, we may assume $B$ has type $\F_{n-1}$ and $H$ acts diagonally on $X^p$ with finitely many orbits for $1 \leq p \leq n-1$. Let $Y$ be a contractible $(B,\eu C)$-CW-complex of type $\F_{n-1}$. As in Theorem \ref{PeterA}, we take $Y$ to have a single orbit of $0$-cells, one of which, $\ast$, is stabilised by $A$.

Let $(Y_\alpha)$ be a directed family of finite subcomplexes of $Y$ which contain the $(n-1)$-skeleton. Let $(W_\alpha)$ be the directed family of complexes $W_\alpha = \oplus_X Y_\alpha$. Note that every $W_\alpha$ has the same $(n-1)$-skeleton as $W=\oplus_X Y$, so they are all $(n-2)$-connected. Now the chain complex of $W$ gives an exact sequence $$0 \to Z_{n-1}(W) \to C_{n-1}(W) \to \cdots \to C_1(W) \to C_0(W) \to \Z \to 0$$ of discrete $G$-modules. As in the proof of Theorem \ref{PeterA}, $G$ acts on the $p$-cells of $W$ with stabilisers commensurable with the stabilisers of the $H$-action on $X^p$; by Corollary \ref{cor:fpinduction} and Lemma \ref{lem:cominvariant}, the submodule of each $C_p(W)$ generated by each orbit of $p$-cells has type $\FP_{n-p}$; because we have finitely many orbits, each $C_p(W)$ has type $\FP_{n-p}$. Also $\Z$ has type $\FP_n$ as a $\Z[G]$-module.

Now, by Corollary \ref{cor:FPn}, $Z_{n-1}(W)$ is finitely generated. For each $\alpha$ we have a short exact sequence $$0 \to B_{n-1}(W_\alpha) \to Z_{n-1}(W) \to H_{n-1}(W_\alpha) \to 0,$$ so $\varinjlim_\alpha H_{n-1}(W_\alpha) = H_{n-1}(W) = 0$ implies there is some $\alpha$ such that $H_{n-1}(W_\alpha) = 0$. Therefore $X_\alpha$ is $(n-1)$-connected. By adding higher cells to kill higher homotopy, we can get a contractible $(B,\eu C)$-CW-complex of type $\F_n$.

Showing that $H$ acts on $X^n$ with finitely many orbits can be done similarly, and is left as an exercise for the reader. This is an adaptation of part of the proof of \cite[Theorem B]{KropMart}: here we write $\oplus_X Y$ as a directed union of subcomplexes $Y^{X_\alpha}$ of $\oplus_X Y$ containing only the $n$-cells whose \emph{grand support} (see \cite[p.8]{KropMart}) is in $X_\alpha \subseteq X^n$.
\end{proof}

In general, we know of no necessary and sufficient conditions on $H$, $B$ and $X$ for $G$ to have type $\F_n$. However, we do have the following important special case, using the same assumption as is employed in \cite[Theorem C]{KropMart}, for example:

\begin{thm}
Suppose $G = B \wr^A_X H$ has type $\F_n$, and there is an epimorphism $f$ from $B$ to $\Z$ with the discrete topology. Then $B$ and $H$ have type $\F_n$ and $H$ acts diagonally on $X^p$ with finitely many orbits and stabilisers of type $\FP_{n-p}$ for $1 \leq p \leq n$.
\end{thm}
\begin{proof}
Thanks to Theorem \ref{PeterB}, we just have to show $H$ acts diagonally on $X^p$ with stabilisers of type $\FP_{n-p}$ for $1 \leq p \leq n$. Note that $f(A)$ is a compact subgroup of $\Z$, so it is trivial. Since $\Z$ is free, $\Z \wr_X H = \Z \wr^{f(A)}_X H$ is a retract of $G$, and hence of type $\F_n$, so we may assume that $B=\Z$.

Clearly $\Z$ has the same homology groups considered as a t.d.l.c. group that it has considered as an abstract group. Together with Bieri's criterion, Theorem \ref{thm:Bieri}, this shows that the proofs of \cite[Proposition 4.1, Corollary 4.2, Theorem C]{KropMart} immediately carry over to our context, and the result follows.
\end{proof}

We sketch briefly a t.d.l.c. analogue of \cite[Theorem D]{KropMart}. Suppose $H$ is poly-(compact-open-by-cyclic). Now for any compact open subgroup $U$ of $H$, $U$ is poly-(cyclic or finite)-ly contained in $H$, in the terminology of \cite{Hall}, so we may apply the proof of \cite[Lemma 3]{Hall}, and the same inductive argument as the Corollary after \cite[Lemma 3]{Hall}, to show: for $N$ a discrete $U$-module contained in a discrete $H$-module $M$ such that $N$ generates $M$ as an $H$-module, if $N$ is noetherian as a $U$-module, $M$ is noetherian as an $H$-module.

In particular, taking $M = \Z[H/U]$ and $N$ the copy of $\Z$ generated by the identity coset of $U$, we get that $\Z[H/U]$ is a noetherian $H$-module. Since this is true for all proper smooth discrete $H$-permutation modules (over $\Z$), we deduce that all finitely generated discrete $H$-modules have type $\FP_\infty$. This fact, together with $H$ having type $\F_\infty$ (by Theorem \ref{thm:LHS}), gives everything we need for the proof of \cite[Theorem D]{KropMart} to hold in our case, and we get:

\begin{thm}
For $H$ poly-(compact-open-by-cyclic), $G = B \wr_X^A H$ has type $\F_n$ if and only if $B$ has type $\F_n$ and $H$ acts diagonally on $X^p$ with finitely many orbits for $p \leq n$.
\end{thm}


\section{Homological dimension vs Cohomological dimension}
\label{hdcd}

A projective resolution $P\to A$ is said to have {\it finite length} if there is a positive integer $n$ such that $P_k=0$ for all $k\geq n$.
The {\it projective dimension} of $A$ is defined to be the minimum $n$ such that $A$ has a projective resolution of finite length $n$. If such an $n$ does not exist we say that $A$ has {\it infinite projective dimension}.

Over a general ring $R$, of course these things are perfectly well-defined. But when $\RG$ does not have enough projectives, it is difficult to know how to make progress: we have nothing like Theorem \ref{fpnkpn} to help us here. So we `search under the streetlight', and see what can be said working over $\Q$. 

Subsequently, the {\it rational discrete cohomological dimension} of $G$, denoted by $\ccd_\Q(G)$, is defined to be the projective dimension of the trivial module $\Q$ over $\QG$. By substituting projective discrete modules with flat discrete modules, the notion of {\it rational discrete flat dimension} arises in $\QGdis$. In particular, a t.d.l.c. group $G$ will be said to have {\it rational discrete homological dimension} $n\in\N\cup\{\infty\}$, denoted by $\hhd_\Q(G)$, if the trivial module $\Q$ has flat dimension $n$. Since every projective discrete $\QG$-module is flat one easily concludes the following.
\begin{fact} Let $G$ be a t.d.l.c. group.
\begin{enumerate}[$(1)$]
\item $\fd_\Q(M)\leq \pd_\Q(M)$ for every $M\in\textup{ob}(\QGdis)$.
\item $\hhd_\Q(G)\leq\ccd_\Q(G)$.
\end{enumerate}
\end{fact}
\begin{prop}
Let $G$ be a t.d.l.c. group. Then:
\begin{enumerate}[$(a)$]
\item every finitely presented flat discrete $\QG$-module is projective;
\item if $M$ is of type $\FP_\infty$, then $\pd_\Q(M)=\fd_\Q(M)$;
\item if $G$ is of type $\FP_\infty$, then $\ccd_\Q(G)=\hhd_\Q(G)$;
\end{enumerate}
\begin{enumerate}[$(a')$]
\item every countably presented flat discrete $\QG$-module has projective dimension less or equal to $1$;
\item if $M$ has a resolution by countably generated permutation $\QG$-modules with compact open stabilisers, then $\pd_\Q(M)\leq \fd_\Q(M)+1$;
\item If $G$ is $\sigma$-compact, then $\ccd_\Q(G)\leq\hhd_\Q(G)+1$.
\end{enumerate}
\end{prop}
\begin{proof} Let $P$ be a finitely generated projective discrete $\QG$-module. Following \cite[\S 4.3]{it:ratdiscoh}, there is an isomorphism
$$\Hom_G(P,\biB(G))\otimes_G M\to\Hom_G(P,M),$$
which is the so-called $\Hom-\otimes$ identity. Thus one can transfer verbatim the proof of \cite[Lemma 4.4(a)]{bieri} and (a) holds. Moreover, (b) and (c) are straightforward consequences of (a).

Since a Lazard-type Theorem is available in the context of discrete $\QG$-modules, one can prove $(a')$ by following the argument of \cite[Lemma~4.4(b)]{bieri}. Subsequently, $(b')$ holds since kernels in a countably generated resolution are countably presented. Finally, when $G$ is $\sigma$-compact, countably generated discrete $\QG$-modules are countable and vice-versa. Therefore let $P\to\Q$ be a countable-type resolution and apply $(b')$.
\end{proof}

\begin{prop}
Let $G$ be a t.d.l.c. group and $N$ a normal closed subgroup. Then $\ccd_\Q(G) \leq \ccd_\Q(N) + \ccd_\Q(G/N)$ and $\hhd_\Q(G) \leq \hhd_\Q(N) + \hhd_\Q(G/N)$.
\end{prop}
\begin{proof}
Apply the Lyndon-Hochschild-Serre spectral sequence.
\end{proof}

\bigskip


\end{document}